\documentclass[11pt, oneside]{article}   	% use "amsart" instead of "article" for AMSLaTeX format

\usepackage[toc,page]{appendix}
\usepackage{geometry}  
\usepackage{hyperref}% See geometry.pdf to learn the layout options. There are lots.
\usepackage{graphicx,wrapfig,lipsum}	
\usepackage[labelformat=empty]{caption}
\usepackage{yfonts}	
\usepackage{amssymb}
\usepackage{amsmath}
\usepackage{amsthm}
\usepackage{scalerel}
\usepackage{verbatim}
\newcommand{\Hplus}{+_{_H}}
\newcommand{\Vplus}{+_{_V}}
\newcommand{\subH}{_{_H}}
\newcommand{\subV}{_{_V}}
\newcommand{\E}{\mathbb{E}}
\DeclareMathOperator*{\hsum}{\scalerel*{\enspace \Sigma_{_H}}{\sum}}
\DeclareMathOperator*{\vsum}{\scalerel*{\enspace \Sigma_{_V}}{\sum}}
\makeatletter
\newtheorem*{rep@theorem}{\rep@title}
\newcommand{\newreptheorem}[2]{%
\newenvironment{rep#1}[1]{%
 \def\rep@title{#2~\ref{##1}}%
 \begin{rep@theorem}}%
 {\end{rep@theorem}}}
\makeatother

\newtheorem{thm}{Theorem}[section]
\newreptheorem{thm}{Theorem}

\newtheorem{lem}[thm]{Lemma}
\newtheorem{prop}[thm]{Proposition}
\newreptheorem{prop}{Proposition}
\newtheorem{cor}[thm]{Corollary}
\newreptheorem{cor}{Corollary}
\newtheorem{conj}[thm]{Conjecture}
\newreptheorem{conj}{Conjecture}

\theoremstyle{definition}
\newtheorem{defn}{Definition}
\newtheorem*{rem}{Remark}

\newtheorem*{ques}{Question}

\title{The Saxl Conjecture for Fourth Powers via the Semigroup Property}
\author{Sammy Luo and Mark Sellke}

\date{}							% Activate to display a given date or no date
\begin{document}
\maketitle

\begin{abstract} The tensor square conjecture states that for $n \geq 10$, there is an irreducible representation $V$ of the symmetric group $S_n$ such that $V \otimes V$ contains every irreducible representation of $S_n$.  Our main result is that for large enough $n$, there exists an irreducible representation $V$ such that $V^{\otimes 4}$ contains every irreducible representation. We also show that tensor squares of certain irreducible representations contain $(1-o(1))$-fraction of irreducible representations with respect to two natural probability distributions. Our main tool is the semigroup property, which allows us to break partitions down into smaller ones. \end{abstract}

\tableofcontents

\newpage

\section{Introduction and Main Results}

Much of the representation theory of the symmetric group $S_n$ is well-understood. Its irreducible representations have known explicit descriptions. One poorly understood facet, however, is the decomposition of tensor products of its representations into irreducibles. This paper focuses on a conjecture related to these decompositions, which was introduced in \cite{pak}.

\begin{conj}[Tensor Square Conjecture]
\label{conj:tsc}
For every $n$ except $2, 4, 9$ there exists an irreducible representation $V$ of the symmetric group $S_n$ such that the tensor square $V \otimes V$ contains every irreducible representation of $S_n$ as a summand with positive multiplicity.
\end{conj}

We first remark that for any faithful representation $V$ of a finite group $G$, i.e. a representation such that each $g\in G$ acts differently on $V$, there is some $n$ such that the tensor power $V^{\otimes n}$ contains every irreducible representation of $G$ (\cite{FH} Ex. 2.37). In the case of $G=S_n$, the standard representation of $S_n$ contains a faithful and irreducible representation $V$ of dimension $n-1$, so a sufficiently large tensor power $V^{\otimes k}$ contains every irreducible representation. However, the tensor square conjecture is a much stronger statement than this because it requires such a small exponent.

As is well known (e.g. \cite{FH}), there is an explicit correspondence between the set of irreducible representations of $S_n$ over $\mathbb{C}$ and the set $P_n$ of partitions $\lambda$ of $n$, i.e. sequences $\lambda=(\lambda_1,\lambda_2,\dots)$ of non-negative integers with $\lambda_1\geq\lambda_2\geq\dots$ and $\sum_{i\geq 1}\lambda_i=n$. By associating $\lambda$ to the Young diagram with $\lambda_i$ boxes in row $i$, we may equivalently correspond each irreducible representation of $S_n$ with a Young diagram with $n$ boxes. This correspondence allows the use of combinatorial tools in analyzing many aspects of the representation theory of $S_n$. Because these notions are equivalent for our purposes, we will freely denote by (e.g.) $\lambda$ both the partition or Young diagram corresponding to $\lambda$ and the associated irreducible representation of $S_n$.

In view of this correspondence, we may express the tensor square conjecture in terms of partitions.

\begin{repconj}{conj:tsc}[Tensor Square Conjecture]
For every $n$ except $2, 4, 9$ there exists a partition $\lambda\vdash n$ such that the tensor square $\lambda \otimes \lambda$ contains every irreducible representation of $S_n$ as a summand with positive multiplicity.
\end{repconj}

This conjecture can also be restated in terms of positivity of \textit{Kronecker coefficients} $g_{\lambda\mu}^{\nu}$, the multiplicities of $\nu$ in the tensor product $\lambda \otimes \mu$: it asserts the existence of $\lambda$ such that $g_{\lambda\lambda}^{\nu}>0$ for all $\nu$. Unlike many other coefficients arising in the representation theory of $S_n$, the Kronecker coefficients lack a known combinatorial interpretation. Indeed, finding one has been said to be ``one of the last major open problems in the ordinary representation theory of the symmetric group" (\cite{quote}). The computation of Kronecker coefficients has been shown to be computationally hard (\cite{IkenHard}).

In \cite{pak}, Pak, Panova, and Vallejo studied the tensor square conjecture and suggested two families of partitions which might satisfy the conjecture, the \textit{staircase} and \textit{caret} partitions. The conjecture that the staircase partition suffices is also known as the \textit{Saxl conjecture} (see \cite{pak,IkenDom}).

\begin{defn}
For $m\geq 1$, the staircase partition $\varrho_m\vdash \binom{m+1}{2}$ is  $$\varrho_m=(m,m-1,\dots,1).$$
\end{defn}

\begin{defn}
For $m\geq 1$, the caret partition $\gamma_m\vdash 3m^2$ is  $$\gamma_m=(3m-1,3m-3,3m-5,\dots, m+3,m+1,m,m-1,m-1,m-2,m-2,\dots, 1,1).$$
\end{defn}

\begin{conj}[Saxl Conjecture]
For $n=\binom{m+1}{2}$, all partitions of $n$ are contained in $\varrho_m^{\otimes 2}$.
\end{conj}

\begin{conj}

For $n=3m^2$, all partitions of $n$ are contained in $\gamma_m^{\otimes 2}$.
\end{conj}

Previous work made progress towards the tensor square conjecture and towards the Saxl conjecture in particular. Pak, Panova, and Vallejo used a lemma on nonzero character values to show that the tensor square of the staircase contains all hooks, partitions with two rows, and some partitions with three rows or with two rows plus an extra column (\cite{pak}). Similar results were shown for the tensor square of the caret shape. They also showed that the staircase partition $\varrho_k$ contains at least $3^{\lceil k/2\rceil -1}$ distinct partitions in its tensor square. This is noteworthy since the total number of partitions of $n=\frac{k(k+1)}{2}$ is also roughly on the order of $e^{ck}$ for some $c$.

Ikenmeyer \cite{IkenDom} further generalized some of this progress by showing a result based on comparability of partitions to the staircase in dominance order. This result, which we will use heavily in our work, will be described in greater detail in Section~\ref{subsec:semigrp}.

Although preliminary evidence suggests that there could in fact be many shapes $\lambda$ that satisfy the tensor square conjecture for each $n$, several simple criteria are known. For example, $\lambda^{\otimes 2}$ contains the alternating representation $1^n$ if and only if $\lambda$ is identical to its conjugate $\lambda'$ \cite{pak}, which means that only symmetric $\lambda$ can satisfy the full tensor square conjecture.

We obtain several results toward the tensor square conjecture. The primary result is the following.

\begin{thm}
For sufficiently large $n$, there exists $\lambda\vdash n$ such that $\lambda^{\otimes 4}$ contains all partitions of $n$.
\label{thm:4thpower}
\end{thm}

The partitions $\lambda$ we use are staircases $\varrho_m$ when $n$ is a triangular number and slight adjustments of them when $n$ is not.

We prove Theorem~\ref{thm:4thpower} by combining two results that are of independent interest: We define a simple metric on the set of partitions of $n$ and show that the set of partitions appearing in $\lambda^{\otimes 2}$ is dense in an appropriate sense. We also show that all partitions close to the trivial representation are contained in $\lambda^{\otimes 2}$. Together, these results almost immediately imply that $\lambda^{\otimes 4}$ contains all partitions of $n$ (when $n$ is sufficiently large).

In addition, we prove some probabilistic results towards the Saxl conjecture. 

\begin{thm}
Let $m\geq 0$. Then $\varrho_m^{\otimes 2}$ contains almost all partitions of $\binom{m+1}{2}$ in the uniform measure, in which all distinct partitions are given the same probability.
\label{thm:unif}
\end{thm}

\begin{thm}
Let $m \geq 0$. Then $\varrho_m^{\otimes 2}$ contains almost all partitions of $\binom{m+1}{2}$ with respect to the Plancherel measure. 
\label{thm:planch}
\end{thm}

\begin{rem}

We use the phrases ``almost all" and ``with high probability" throughout this paper to mean that a sequence of probabilities tends to $1$ as the parameter $n$ or $m$ tends to infinity. For example, Theorem~\ref{thm:unif} states that the probability $p(m)$ that a uniformly random partition $\lambda\vdash\binom{m}{2}$ is contained in $\varrho_m^{\otimes 2}$ converges to $1$ as $m$ grows large. In particular, this language implies nothing about the rate of convergence. 

\end{rem}

The Plancherel measure will be discussed in Section~\ref{subsec:measure}.

\section{Basic Tools}
\subsection{Partitions and Representations}
\label{subsec:basics}

Recall that irreducible representations of the symmetric group $S_n$ correspond to Young diagrams, or equivalently to partitions $\lambda \vdash n$ of $n$. We will use Young diagrams and partitions interchangeably, and we denote by $P_n$ the set of partitions of $n$. We will use multiple orientations on Young diagrams: as is conventional, we call the coordinates depicted below at left \textit{English}, those in the middle \textit{French}, and those at right \textit{Russian}. We use $1_n$ to denote the trivial representation, $1^n$ the alternating representation, and $\lambda '$ the conjugate of $\lambda$.

%%testing

\begin{figure}[ht]
\label{fig:Orientation}
\centering
\begin{minipage}[b]{0.3\linewidth}
\includegraphics[scale=0.4]{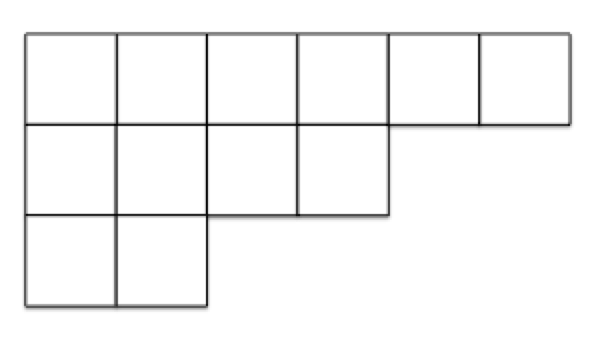}
\caption{English Coordinates}
\label{fig:minipage1}
\end{minipage}
\quad\begin{minipage}[b]{0.3\linewidth}
\includegraphics[scale=0.4]{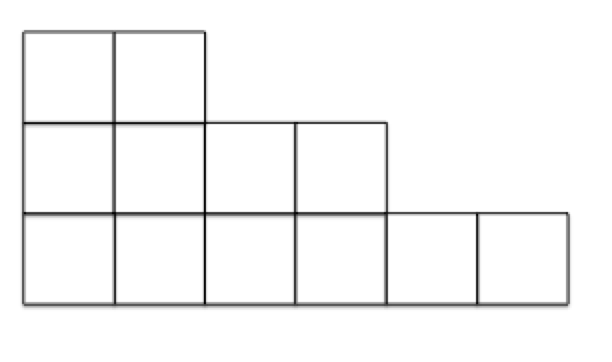}
\caption{French Coordinates}
\label{fig:minipage2}
\end{minipage}
\quad
\begin{minipage}[b]{0.3\linewidth}
\includegraphics[scale=0.4]{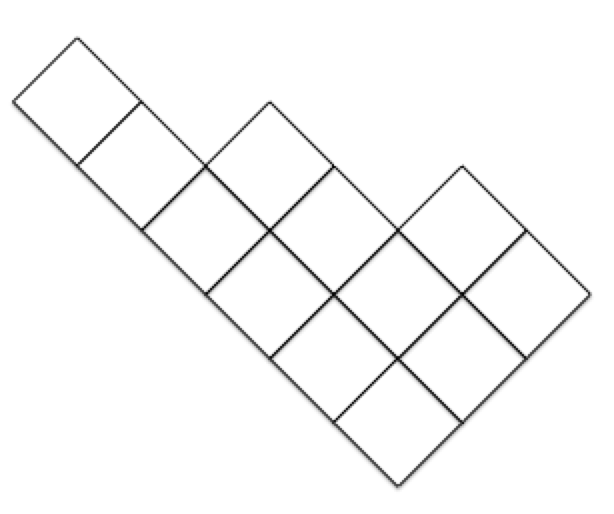}
\caption{Russian Coordinates}
\label{fig:minipage3}
\end{minipage}
\end{figure}

\begin{lem} [{\cite[Section 4.1]{FH}}] Let $\lambda \vdash n$. Then $\lambda' = \lambda \otimes 1^n$.
\label{lem:conjalt}
\end{lem}

For simplicity, we define an indicator function for constituency.
 
\begin{defn}
Let $c(\lambda,\mu,\nu)$ be the statement that  $g_{\lambda,\mu}^{\nu}>0$.
\end{defn}

We first establish some well-known, basic properties of $c(\lambda,\mu,\nu)$.

\begin{lem} The function $c$ is symmetric in its three arguments.

\end{lem}
\begin{proof}

Let $\chi^V:S_n\to\mathbb C$ denote the character function for a representation $V$. We have
$$g_{\lambda,\mu}^{\nu}=\langle \chi^{\lambda}\chi^{\mu},\chi^{\nu}\rangle =\langle \chi^{\lambda},\chi^{\mu}\chi^{\nu}\rangle=g_{\mu,\nu}^{\lambda},$$
and similarly for other permutations. Here we use the facts that all representations of $S_n$ have real characters and that $\chi^{V\otimes W}=\chi^V\chi^W$.

\end{proof}
\begin{lem}
\label{lem:conj2}

If $c(\lambda,\mu,\nu)$ then $c(\lambda',\mu',\nu)$.

\end{lem}
\begin{proof}
By Lemma~\ref{lem:conjalt} we have $\lambda'\otimes\mu'= (\lambda\otimes 1^n) \otimes (\mu\otimes 1^n)=(\lambda\otimes \mu) \otimes (1^n\otimes 1^n) = \lambda\otimes\mu$, so the result follows by the definition of the Kronecker coefficient.
\end{proof}

There is one more special representation that we will use repeatedly.

\begin{defn}
The \textit{standard representation} of $S_n$ is $\tau_n=\operatorname{Ind}_{S_{n-1}}^{S_n}(1)$. Equivalently, $\tau_n$ is the $n$-dimensional representation in which $S_n$ acts by permuting $n$ basis vectors $v_1\dots v_n$ in the usual  way.

\end{defn}

It is well known that $\tau_n$ is the sum of the irreducible representations corresponding to the partitions $(n), (n-1,1)$; the trivial component comes from the vector $(v_1+\dots+v_n)$, which all elements of $S_n$ fix. It is easy to see from the above description that the remaining part corresponding to $(n-1,1)$ is a faithful representation, as claimed above.

It is known that for an irreducible representation $\lambda$, the tensor product $\tau_n \otimes \lambda$ is the formal sum (with multiplicity) of all partitions which can be formed by moving a single square in the Young diagram for $\lambda$ (including $\lambda$ itself). (Pieri's Rule, \cite{FH} Ex. 4.44). This fact will be used extensively later in the paper. 

Finally, we recall the definition of the Durfee square.

\begin{defn}
The \textit{Durfee length} $d(\lambda)$ of a partition $\lambda$ is the largest integer $r$ with $\lambda_r\geq r$.

\end{defn}

\begin{defn}
The \textit{Durfee square} of a partition $\lambda$ is the square of side length $d(\lambda)$ with the principal diagonal as its diagonal  (beginning in the upper-left corner in English coordinates), considered as a subset of $\lambda$ in the plane.

\end{defn}

\subsection{The Semigroup Property and Dominance Ordering}
\label{subsec:semigrp}

We extensively use the $\textit{semigroup property}$, which was proved in \cite{semigp}. To state this, we first define the \textit{horizontal sum} of partitions, in which we add row lengths, or equivalently take the disjoint union of the multisets of column lengths.

\begin{defn}
The \textit{horizontal sum} $\lambda\Hplus\lambda_2$ of partions $\lambda=(\lambda_1,\lambda_2, \dots ) \vdash n_1$ and $\mu =(\mu_1, \mu_2 \dots) \vdash n_2$ is the partition $(\lambda_1+\mu_1, \lambda_2+\mu_2,  \dots ) \vdash (n_1+n_2)$.
\end{defn}

We also define horizontal scalar multiplication by positive integers on partitions, simply by repeated addition.

\begin{defn}
For $k\geq 0$, define the horizontal scalar multiple $k\subH\lambda$ by $k\subH\lambda=\lambda\Hplus\lambda\Hplus\dots\Hplus \lambda$, where we add $k$ copies of $\lambda$. 
\end{defn}

We also define vertical addition and scalar multiplication analogously, by adding column lengths instead.

\begin{defn}
We define the \textit{vertical sum} $\lambda_1\Vplus\lambda_2$ of $\lambda_1,\lambda_2$ to be $(\lambda_1'\Hplus\lambda_2')'$
\end{defn}

\begin{defn}
Define the vertical scalar multiple $k\subV\lambda$ by $k\subV\lambda=\lambda\Vplus\lambda\Vplus\dots \lambda$ where we add $k$ copies of $\lambda$. 
\end{defn}

We now state the \textit{semigroup property}.

\begin{thm} [Semigroup Property, {\cite[Theorem~3.1]{semigp}}]
\label{thm:hsum}
If $c(\lambda_1, \lambda_2, \lambda_3)$ and $c(\mu_1,\mu_2, \mu_3)$ then $c(\lambda_1\Hplus\mu_1, \lambda_2\Hplus\mu_2, \lambda_3 \Hplus \mu_3).$

\end{thm}

We can use induction on Theorem~\ref{thm:hsum} to extend the semigroup property to arbitrary numbers of partitions. However, we will not need this for the bulk of our paper, so we refer the reader to Appendix~\ref{subsec:rectcube}. 

We now give a modified version of \ref{thm:hsum} using vertical sums.

\begin{cor} If $c(\lambda_1, \lambda_2, \lambda_3)$ and $c(\mu_1,\mu_2, \mu_3)$, then $c(\lambda_1\Vplus\mu_1, \lambda_2\Vplus\mu_2, \lambda_3 \Hplus \mu_3).$

\begin{proof}

By Lemma~\ref{lem:conj2} we have $c(\lambda'_1,\lambda'_2,\lambda_3)$ and $c(\mu'_1,\mu'_2,\mu_3)$. Then the semigroup property yields $c(\lambda'_1\Hplus \mu'_1,\lambda'_2\Hplus \mu'_2,\lambda_3\Hplus \mu_3)$. Applying Lemma~\ref{lem:conj2} again yields the result.

\end{proof}

\end{cor}

In other words, in using the semigroup property we are allowed to use an even number of vertical additions in each step. It is \textit{not} true that vertically adding all 3 partitions preserves constituency. For example, we have $c((1),(1),(1))$ for the trivial representations of $S_1$, but vertically adding this to itself gives that the alternating representation of $S_2$ is contained in its own tensor square. This tensor square is just the trivial representation which, of course, does not contain the alternating representation.

We will also extensively use the following result from \cite{IkenDom}. First we recall the notion of dominance ordering, which gives a partial ordering on partitions of $n$.

\begin{defn}
Let $\lambda, \mu \vdash n$. We say that $\lambda$ \textit{dominates} $\mu$ (or $\lambda \succeq \mu$) if for all $k \geq 1$, $\sum_{i=1}^{k}\lambda_i \geq \sum_{i=1}^{k}\mu_i$. 
\end{defn}

\begin{thm} [{\cite[Thm~2.1]{IkenDom}}]
\label{thm:dominance}
Let $m\geq 1$. Then $\varrho_m \otimes \varrho_m$ contains all partitions $\lambda$ which are dominance-comparable to $\varrho_m$.
\end{thm}

\begin{rem}
\label{rem:gendom}
We have also generalized Theorem~\ref{thm:dominance} to arbitrary partitions with distinct row lengths (see Theorem~\ref{thm:gendom}), which seems potentially useful for extending the applicability of the semigroup property.
\end{rem}

The main method used throughout the paper will be to try to express triples $(\varrho_k,\varrho_k,\lambda)$ as sums of smaller triples, each of which satisfies constituency because of Theorem~\ref{thm:dominance}, and then conclude $c(\varrho_k,\varrho_k,\lambda)$ via the semigroup property. This method is powerful because small staircases can be added together to form larger staircases. This will be explained throughout the following sections, which contain overviews of the proofs of our main results.

\section{Overview of the Probabilistic Approach}
\subsection{Partition Measures}
\label{subsec:measure}

In this section, we address a probabilistic weakening of the tensor square conjecture:

\begin{ques}
For large $m$, what is the probability that a random partition of $\binom{m+1}{2}$ is a constituent in the tensor square $\varrho_m^{\otimes 2}$?
\end{ques}

To answer this question, we must first put a probability distribution on the set $P_n$ of partitions of $n$. The most obvious choice is the uniform distribution.

\begin{defn}
\label{defn:unif}
The \textit{uniform measure} $U_n$ assigns probability $\frac{1}{|P_n|}$ to each distinct partition of $n$.
\end{defn}

There is another natural family of probability distributions on $P_n$ we investigate, which is rooted in representation theory.

\begin{defn}
\label{defn:planch}
The \textit{Plancherel measure} $M_n$ assigns to each $\lambda \vdash n$ probability $\frac{\dim (\lambda)^2}{n!}$. Here $\dim (\lambda)$ is the dimension of $\lambda$ as a representation of $S_n$.
\end{defn}

The value $\dim(\lambda)$ has the following famous combinatorial interpretation: it is the number of standard Young tableau of shape $\lambda$, i.e. bijective assignments of $(1,2,\dots,n)$ to the boxes of $\lambda$ such that the numbers increase along each row and column \cite{enumcombo}. \\

\subsection{Limit Shapes of Partition Measures} 

\label{subsec:limshape}
The uniform and Plancherel measures $U_n$ and $M_n$  give rise to different smooth limit shapes for large $n$; in each case, we may speak of the ``typical shape" of a large random partition. Given a Young diagram $\lambda$ of size $n$, we may shrink it by a linear-scale factor of $\frac{\sqrt{2}}{\sqrt{n}}$ so that it has area 2, and rotate it into Russian (diagonal) coordinates. This results in the graph of a function $f_{\lambda}$, and by defining $f_{\lambda}(x)=|x|$ past the boundary of the Young diagram, we get a function defined on the whole real line which satisfies 

$$f_{\lambda}(x) \geq |x|, \hspace{1cm} (1)$$
and
$$\int_{-\infty}^{\infty} (f_{\lambda}(x)-|x|) dx = 2,\hspace{1cm} (2)$$
and
$$|f_{\lambda}(x)-f_{\lambda}(y)| \leq |x-y|.\hspace{1cm} (3)$$

\begin{defn} The set $\mathcal{CY}$ of \textbf{continuous Young diagrams} is the set of functions $\mathbb{R}\rightarrow\mathbb{R}$ satisfying $(1), (2),$ and $(3)$.
\end{defn}

We also define a family $\mathcal{SY}$ of continuous Young diagrams in French coordinates, this time with area $1$. Given $f_1, f_2 \in \mathcal{CY}$, we change them into $F_1, F_2$ in French coordinates by simply rotating and then reflecting, and also dilating by a linear scale factor of $\frac{1}{\sqrt{2}}$. Taking the c\`adl\`ag version, we obtain a non-increasing function $(0,\infty) \rightarrow [0,\infty)$. 

\begin{defn}

Let $\mathcal{SY}$ be the set of functions $(0,\infty) \rightarrow [0,\infty)$ which are non-increasing, c\`adl\`ag, and have total integral 1.

\end{defn}

\begin{defn}
For $f_1 \in \mathcal{CY}$ , the \textit{straightening} $F_1 = S(f_1) \in \mathcal{SY}$ is the right-continuous function $\mathbb{R}^+ \rightarrow \mathbb{R}^+$ given by rotation to English coordinates of the graph of $f_1$ followed by reflection over the x-axis, and then dilation by a factor of $\frac{1}{\sqrt{2}}$.

\end{defn}

\begin{defn}

For $\lambda\vdash n$, define $F_{\lambda}=S(f_{\lambda})$.

\end{defn}

On $\mathcal{CY}$ we will use the supremum norm $d(f,g)=||f-g||_{\infty}$. We define the metric on $\mathcal{SY}$ to agree with $d$ under the canonical bijection $S$. This makes $d$ on $\mathcal{SY}$ simply twice the L\'{e}vy metric (\cite{Prob}).

Using the canonical map $P_n \rightarrow \mathcal{CY}$ given by $\lambda \mapsto f_{\lambda}$, we may pushforward the measures $M_n, U_n$ to finitely supported measures on $\mathcal{CY}$. This allows us to state the limit shape results precisely.

\begin{thm} [\cite{PlanchLimShape}]
\label{thm:planchlimshape}
The pushforwards of the measures $M_n$ converge in probability to the delta measure on the function $$f_M(x)=\left\{
     \begin{array}{lr}
       (\frac{2}{\pi})(x\arcsin(\frac{x}{2})+\sqrt{4-x^2}) & : x \leq 2,\\
       |x| & : |x| \geq 2.
     \end{array}
   \right.$$

\end{thm}

\begin{thm} [\cite{UnifLimShape}]
\label{thm:uniflimshape}
The pushforwards of the measures $U_n$ converge in probability to the delta measure on the function $$f_U(x)=\frac{1}{b}\log(e^{-xb}+e^{xb})$$
where $b=\frac{\pi}{2\sqrt{6}}$.
\end{thm}

So under both uniform and Plancherel measure, for large $n$, almost all Young diagrams look like a smooth limiting curve if we zoom out to a constant-size scale. The limit shapes for the Plancherel and uniform measures are (respectively) shown below.

\begin{figure}[ht]
\label{fig:PlanchLimShape}
\centering
\begin{minipage}[b]{0.5\linewidth}
\includegraphics[scale=.5]{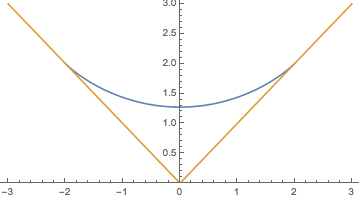}
\caption{Plancherel Limit Shape}
\label{fig:UnifLimShape}
\end{minipage}
\quad
\begin{minipage}[b]{0.3\linewidth}
\includegraphics[scale=.5]{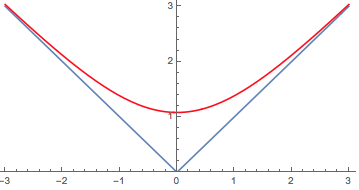}
\caption{Uniform Limit Shape}
\end{minipage}
\end{figure}

The limit shape of $U_n$ may be more elegantly described in French coordinates, where it is the graph of $e^{-\frac{\pi x}{\sqrt{6}}}+ e^{-\frac{\pi y}{\sqrt{6}}}=1$.

We will extend the definitions of length and height of Young diagrams to the continuous Young diagrams. For ordinary Young diagrams $\lambda \vdash n$, we have that the length of $\lambda$ is $\lambda_1=\sqrt{n}\sup(\{x|f_{\lambda}(x)=-x\})$ and the height is $\lambda'_1=\sqrt{n}\inf(\{x|f_{\lambda}(x)=x\})$, given by the lowest spots on the rotated axes that the functions meet. Defining the normalized length and height by $\ell(\lambda):=\frac{\lambda_1}{\sqrt{n}}, h(\lambda):=\frac{\lambda'_1}{\sqrt{n}}$, we may extend the rescaled length and height to $\mathcal{CY}$ and hence also to $\mathcal{SY}$:

\begin{defn}
For $f \in \mathcal{CY}$, define its (possibly infinite) \textbf{length} and \textbf{height} as
$$\ell(f)=-\inf(\{x|f(x)=-x\}),$$
$$h(f)=\sup(\{x|f(x)=x\}).$$

\end{defn}

\begin{defn}

For $F\in\mathcal{SY}$, let $\ell(F)=\inf(\{x|F(x)=0\})$.

\end{defn}

\begin{defn}

For $F\in\mathcal{SY}$, let $h(F)=\lim_{t\to 0^+} F(t)$.

\end{defn}

Note that these definitions respect the bijection $S$ because of the area difference between the types of continuous Young diagrams.

We define $C_f$ to be the region between $f \in \mathcal{CY}$ and the graph of $y=|x|$, and $S_F$ the corresponding notion for $F\in\mathcal{SY}$.

\begin{defn}

For $f\in\mathcal{CY}$, let $C_f=\{(x,y)|f(x)\geq y\geq |x|\}$.

\end{defn}

\begin{defn}

For $F\in\mathcal{SY}$, let $S_F=\{(x,y)\in(\mathbb{R}^+)^2|y\leq F(x)\}$.

\end{defn}

We now extend the dominance ordering to continuous Young diagrams.

\begin{defn}

For $F \in\mathcal{SY}$, define $H_F(t)=\int_0^{t} F(x) dx $.

\end{defn}

\begin{defn}
\label{defn:contdomin}
For $F_1, F_2 \in \mathcal{SY}$ we say that $F_1 \succeq F_2$ if they satisfy

$$H_{F_1}(t) \leq H_{F_2}(t)$$
for all $t$.
\end{defn}

\begin{defn}

For $f_1, f_2 \in\mathcal{CY}$ we say that $f_1\succeq f_2$ if $S(f_1)\succeq S(f_2)$.

\end{defn}

It is clear that these dominance orders are extensions of the ordinary dominance order. That is, if $\lambda,\mu\vdash n$ then $\lambda \succeq \mu$ if and only if $f_{\lambda}\succeq f_{\mu}$.

We also define $\varrho_{\infty}$.

\begin{defn}

Let $\varrho_{\infty}$ be the continuous Young diagram which is an isosceles right triangle, so \[\lim_{k\rightarrow\infty}\varrho_k=\varrho_{\infty}.\]
\end{defn}

We now establish a few simple lemmas on the geometry of continuous Young diagrams. 

\begin{defn}
For $A$ a subset of $\mathbb{R}^2$, let $A^{\varepsilon}=\{x\in\mathbb{R}^2|d(x,A)\leq \varepsilon\}$ and $A_{\varepsilon}=((A^{C})^{\varepsilon})^C$, where $A^C$ is the complement of $A$.
\end{defn}

\begin{defn}

For $A\subseteq \mathbb{R}^2$, we denote by $m(A)$ the Lebesgue measure of $A$ (assuming it exists).

\end{defn}

\begin{lem}
\label{lem:semicont}

The functions $\ell, h \colon \mathcal{CY}\to \mathbb{R}$ are lower semi-continuous with respect to the norm $d$. 

\end{lem}

\begin{proof}

We show the result for $h$, the other case being exactly the same. Suppose $h(f) > A$. We show that for $g$ sufficiently close to $f$, we have $h(g) >A$. Indeed, for some $x>A$ we have $f(x)-|x|=c>0$. For $|f-g| < c$, we have $g(x)-|x|>0$, implying $h(g)>A$ as desired. 

\end{proof}

\begin{lem}
\label{lem:convmeas}
For fixed $F\in\mathcal{SF}$, if $S_F$ is a bounded subset of the plane then we have 

$$ \left(\lim_{\varepsilon\to 0^+}{m((S_F)^{\varepsilon})}\right) = \left( \lim_{\varepsilon\to 0^+}{m((S_F)_{\varepsilon})}\right) =m(S_F)=1. $$

\end{lem}

\begin{proof}

It is clear that we have $m(S_F)=1$. We proceed with the remaining claims.

We note that the closure $\overline{S_F}$ and the interior $\overset{\circ}{S_F}$ have equal measure. Indeed, $\overline{S_F}-(\overset{\circ}{S_F})$ consists of a countable number of vertical lines (corresponding to the discontinuities of $F$) and the graph of $F$, each of which has measure 0 by Fubini's theorem. Now, as $\varepsilon$ approaches $0$, the sets $(S_F)^{\epsilon}$ intersect to $\overline{S_F}$. Since $S_F$ is bounded, these sets converge in measure to $\overline{S_F}$ by the dominated convergence theorem. Similarly, as $\varepsilon$ approaches $0$ the sets $(S_F)_{\epsilon}$ converge in measure to $\overset{\circ}{S_F}$. Since we have $m(\overline{S_F})=m(S_F)=m(\overset{\circ}{S_F})$, the lemma follows.

\end{proof}

\begin{lem}
\label{lem:shapedomin}
Given $f_1, f_2 \in \mathcal{CY}$, assume that $f_1(x)=f_2(x)>|x|$ has at most 1 solution in $x$ and that $\ell(f_1)>\ell(f_2),$ $h(f_1) < h(f_2)$. Then $f_1(x)=f_2(x)>|x|$ has a unique solution and $f_1 \succeq f_2$. Moreover, there exists $\varepsilon>0$ such that the following holds: for all continuous Young diagrams $g_1, g_2$ with $||g_i-f_i||_{\infty} \leq \varepsilon$ for $i \in \{1,2\}$ and also $|\ell(f_2)-\ell(g_2)| \leq \varepsilon$ and $|h(f_1)-h(g_1)| \leq \varepsilon$, we have $g_1\succeq g_2$.

\end{lem}

\begin{proof}[Proof of Lemma~\ref{lem:shapedomin}]

The following pictures will probably be helpful aids for understanding the proof. They correspond to Russian and French coordinates, respectively, with the blue curve as $f_1$ and $ F_1=S(f_1)$ and the green line as $f_2$ and $ F_2=S(f_2)$.

\begin{figure}[ht]
\centering
\begin{minipage}[b]{0.5\linewidth}
  \includegraphics[scale=.5]{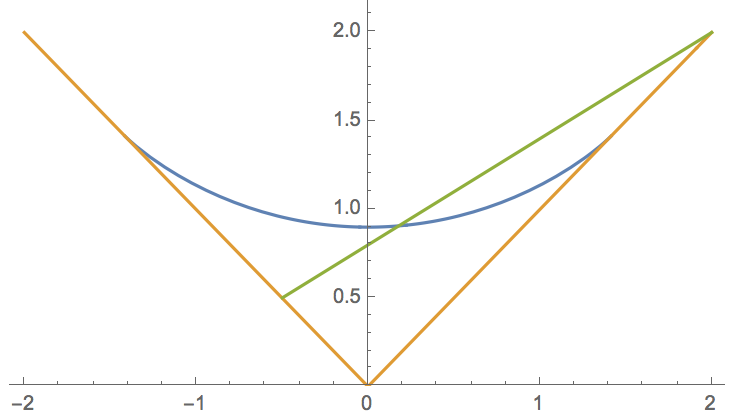}
\caption{Continuous Young Diagrams in English Coordinates}
\label{fig:ContYoungRuss}
\end{minipage}
\quad
\begin{minipage}[b]{0.45\linewidth}
\includegraphics[scale=.5]{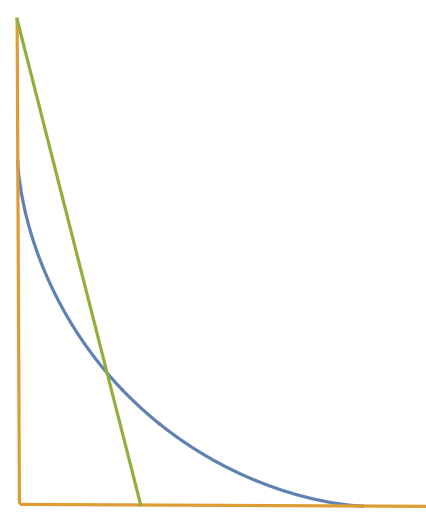}
\caption{Continuous Young Diagrams in Russian Coordinates}
\label{fig:ContYoungFrench}
\end{minipage}
\end{figure}

First we establish that 

$$f_1(x)=f_2(x)>|x|$$
has a solution. This is simple: we have $f_1(-\ell(f_2)) > f_2(-\ell(f_2))=\ell(f_2)$ and $f_2(h(f_1))>f_1(h(f_1))=h(f_1))$, so we may apply the intermediate value theorem to find a real number $c \in (-\ell(f_2),h(f_1))$ for which $f_1(c)=f_2(c)$. By the Lipschitz condition on $\mathcal{CY}$, $c$ clearly satisfies $f_1(c)=f_2(c)>|c|.$

Consider the continuous function $f_3(x)=f_1(x)-f_2(x)$. We have that the equation  

$$f_1(x)=f_2(x)>|x|$$
has a unique solution $c$, and also that $\{x|f_1(x)=f_2(x)=|x|\}=(-\infty,-\ell(f_1))\cup (h(f_2),\infty).$ Therefore, $f_3^{-1}(0)=(-\infty,-\ell(f_1))\cup \{c\} \cup (h(f_2),\infty)$. We have again that $f_1(-\ell(f_2)) > f_2(-\ell(f_2))=\ell(f_2)$ and $f_2(h(f_1))\geq f_1(h(f_1))=h(f_1)$, which yield $f_3(-\ell(f_2))\geq 0$ and $f_3(h(f_1))\leq 0$. By continuity, we therefore see that $f_3(x)>0$ for $x\in  (-\ell(f_1),c)$ and $f_3(x)<0$ for $x\in (c,h(f_2))$.

Let us straighten this picture into French coordinates, with $F_1=S(f_1), F_2=S(f_2)$ and $(c,f_1(c))$ becoming $(k, F_1(k))$. Then our last deduction translates into the fact that $F_2-F_1$ is positive on $(0,k)$, 0 at $k$, negative on $[k,\ell(f_1))$ and 0 on $[\ell(f_1),\infty)$. Since both functions have integral 1, is it clear to see that the function $I(t)=H_{F_2}(t)-H_{F_1}(t)$ is 0 at 0, positive on $(0,\ell(f_1))$ and 0 on $[\ell(f_1),\infty)$. Therefore $f_1\succeq f_2$. 

We now establish $g_1\succeq g_2$ for $g_i$ satisfying the conditions in the lemma statement for small $\varepsilon$. Define $G_1, G_2 \in \mathcal{SY}$ as $S(g_1),S(g_2)$, and $J(t)=H_{G_2}(t)-H_{G_1}(t)$. We show $J(t)\geq 0$ for all $t$, when $\varepsilon$ is small enough. First, the condition $|\ell(f_2)-\ell(g_2)| \leq \varepsilon$ combined with $\ref{lem:semicont}$ guarantees that $J(t)$ is positive near $\ell(f_1)$, because for small enough $\varepsilon$ we have $\ell(G_1)>\ell(G_2)$. Similarly, the condition $|h(f_1)-h(g_1)| \leq \varepsilon$ combined with $\ref{lem:semicont}$ ensures that the initial ``columns" of $G_2$ are larger than than those of $G_1$, i.e. that $G_2(x)-G_1(x)>0$ for $x$ small. So for $\varepsilon$ small, $J(t)$ is positive near 0. 

Now we need only show that $J(t)\geq 0$ for $t\in [\alpha,\ell(f_1)-\alpha]$ for some $\alpha > 0$. Note that because $I(t)>0$ on $[\alpha,\ell(f_1)-\alpha]$, by compactness there exists $\delta > 0$ with $I(t) > \delta$ for $t \in [\alpha,\ell(F_1)-\alpha]$. 

Now, we claim that $(S_{F_2})_{\varepsilon \sqrt{2}} \subseteq S_{G_2}$. Indeed, if $(x,y) \in (S_{F_2})_{\varepsilon \sqrt{2}}$
then 

$$(x+\varepsilon,y+\varepsilon)\in S_{F_2} \iff y+\varepsilon \leq F_2(x+\varepsilon) \implies y \leq G_2(x),$$ 
as desired. Therefore, by Lemma $\ref{lem:convmeas}$, picking $\varepsilon$ small forces $S_{F_2}, S_{G_2}$ to be arbitrarily close in measure. We have 

$$|J(t)-I(t)| \leq m((S_{F_2}\Delta S_{G_2}) \cap ([0,t]\times \mathbb{R})) \leq m(S_{F_2}\cap S_{G_2}), $$
so by picking $\varepsilon$ small we force $|J(t)-I(t)|$ to be uniformly small. Since $I(t)$ is bounded away from $0$ on the interval $[\alpha,\ell(F_1)-\alpha]$, for small $\varepsilon$ we have $J(t)>0$ on that interval. This shows we can force $J(t)>0$ for all $t\in\mathbb{R}$, concluding the proof.

\end{proof}

\begin{lem}
\label{lem:shapearea}

For $F \in \mathcal{SY}$ let $H_F(t)=\int_{0}^t F(t) dt$. For every $\varepsilon > 0$ there exists $\delta$ such the the following holds: if $G \in \mathcal{SY}$ satisfies $d(F,G) \leq \delta$ then $|H_F-H_G|_{\infty} \leq \epsilon$. 

\end{lem}

\begin{proof}[Proof of Lemma~\ref{lem:shapearea}]

Fix $\varepsilon$ as in the statement. By exhaustion take $a,$ $b$ such that $\int_a^b F(t) dt \geq 1-\frac{\varepsilon}{4}$. Now pick $\delta$ such that if $d(F,G)\leq\delta$ then $|F(x)-G(x)| \leq \frac{\varepsilon}{4(b-a)}$ for $x \in [a,b]$. Then $\int_a^b G(t) dt \geq 1-\frac{\varepsilon}{2}$ and so $\int_0^a G(t) dt \leq \frac{\epsilon}{2}, 
\int_b^{\infty} G(t) dt \leq \frac{\epsilon}{2}$ because $\int_0^{\infty} G(t) dt =1$. Therefore, for $x \not\in [a,b]$ the desired 
$|H_F(x)-H_G(x)| \leq \epsilon$ holds. For $x \in [a,b]$ we have $|H_F(x)-H_G(x)| \leq |H_F(a)-H_G(a)|+|\int_a^x F(x)-G(x)| \leq \frac{\varepsilon}{2}+\frac{\varepsilon}{2}= \epsilon$.

\end{proof}

\subsection{Overview of Proofs for Probabilistic Results}

For both $U_n$ and $M_n$, we know now the overall shape of a generic partition $\lambda \vdash n$. Our strategy is the following: we will decompose these partitions into pieces, handle each piece using Theorem~\ref{thm:dominance} on dominance, and combine these pieces using the semigroup property. We will need a way to combine smaller staircases into larger ones. For this section, we will use the following identities.

\begin{prop}
We have the identities
$$(\varrho_{k} \Hplus \varrho_{k-1}) \Vplus (\varrho_k \Hplus \varrho_k) = \varrho_{2k}, \hspace{0.5cm} (\varrho_{k+1} \Hplus \varrho_{k}) \Vplus (\varrho_k \Hplus \varrho_k) = \varrho_{2k+1}.  $$
\end{prop}

Visual depictions for both cases when $k=2$ are below. \\

\begin{figure}[ht]
\centering
\begin{minipage}[b]{0.5\linewidth}
 \includegraphics[scale=0.4]{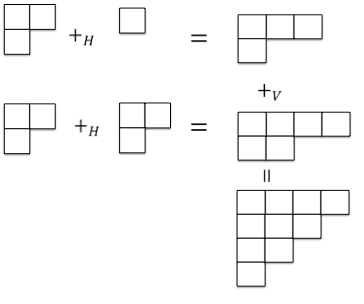}
\caption{Decomposition of $\varrho_4$.}
\label{fig:stairsum1}
\end{minipage}
\quad
\begin{minipage}[b]{0.45\linewidth}
\includegraphics[scale=0.4]{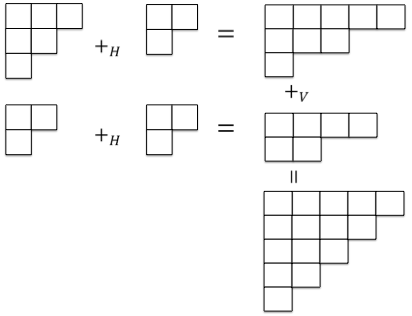}
\caption{Deomposition of $\varrho_5$.}
\label{fig:stairsum2}
\end{minipage}
\end{figure}

This means that we can break a large staircase of size $n$ into four staircases of size roughly $\frac{n}{4}$. Supposing for convenience that $m=2k$ is even, by the above proposition, to show $c(\varrho_m,\varrho_m,\lambda)$ for some $\lambda$, it suffices to write $\lambda$ as $$\lambda=\lambda_1\Hplus\lambda_2\Hplus\lambda_3\Hplus\lambda_4$$
where $c(\varrho_{k-1},\varrho_{k-1},\lambda_1)$ and $c(\varrho_k,\varrho_k,\lambda_i)$ for $i \in \{2,3,4\}$. For in such a case, the semigroup property gives 
$$c(\varrho_{k-1}\Hplus\varrho_k,\varrho_{k-1}\Hplus\varrho_k,\lambda_1\Hplus\lambda_2), \hspace{1cm} c(\varrho_{k}\Hplus\varrho_k,\varrho_{k}\Hplus\varrho_k,\lambda_3\Hplus\lambda_4)$$
$$\implies c((\varrho_{k-1}\Hplus\varrho_k)\Vplus(\varrho_k\Hplus\varrho_k),(\varrho_{k-1}\Hplus\varrho_k)\Vplus(\varrho_k\Hplus\varrho_k),(\lambda_1\Hplus\lambda_2)\Hplus(\lambda_3\Hplus\lambda_4))$$
$$\implies c(\varrho_{2k},\varrho_{2k},\lambda).$$

%\begin{wrapfigure}{R}{3cm}
%\vspace{-10pt}
%\caption{An Equal Area Decomposition of the Uniform Limit Shape into 4 Parts by Column Size} \label{}
%\includegraphics[scale=0.7]{"UnifCut".png}
%\end{wrapfigure} 

Perhaps surprisingly, this method suffices to show constituency in $\varrho_m^{\otimes 2}$ for almost all partitions in both the uniform and Plancherel cases. In each case, we can break the limit shape into four pieces of approximately equal area. For the uniform shape, we use vertical cuts as depicted at right. For the Plancherel limit shape, we evenly distribute the columns among the four smaller pieces, so that each smaller piece has approximately the same shape.

In each case, the dominance comparability of the smaller pieces with a staircase partition follows by geometric reasoning on the limit shapes. What is a bit more subtle is ensuring that these decompositions can be done precisely: to apply the semigroup property, each of our four pieces must have size exactly equal to that of a certain staircase partition. For this point, our strategy is to divide almost all of the large partition into our four pieces but reserve some very short columns for the end of the process. We distribute these short columns to the pieces such that each piece has exactly the correct size, without affecting their overall shapes significantly. To justify all of these details requires a bit of care, and we do this in the following section.

\section{Proofs of Probabilistic Results}
\label{sec:probprf}

\subsection{The Uniform Case}

\begin{repthm}{thm:unif}
Let $m\geq 0$. Then $\varrho_m^{\otimes 2}$ contains almost all partitions of $\binom{m+1}{2}$ in the uniform measure, in which all distinct partitions are given the same probability.
\end{repthm}

\begin{proof}

Our plan is to use the semigroup property to add together smaller cases that can be proven by dominance ordering. 

We use Proposition~\ref{thm:stairgrid} in the $k=2$ case, to give the following: Suppose we have partitions $\lambda_1, \lambda_2,\lambda_3,\lambda_4$ which are dominance comparable to $\varrho_{\lfloor \frac{m+1}{2}\rfloor}, \varrho_{\lfloor \frac{m}{2}\rfloor},\varrho_{\lfloor \frac{m}{2}\rfloor},\varrho_{\lfloor \frac{m-1}{2}\rfloor}$, respectively. Repeated application of the semigroup property yields 
$$c\left(\varrho_{{\lfloor \frac{m+1}{2}\rfloor}}\Hplus \varrho_{\lfloor \frac{m}{2}\rfloor},\varrho_{{\lfloor \frac{m+1}{2}\rfloor}}\Hplus\varrho_{\lfloor \frac{m}{2}\rfloor},\lambda_1\Hplus\lambda_2\right), c\left(\varrho_{{\lfloor \frac{m}{2}\rfloor}}\Hplus \varrho_{\lfloor \frac{m-1}{2}\rfloor},\varrho_{{\lfloor \frac{m}{2}\rfloor}}\Hplus\varrho_{\lfloor \frac{m-1}{2}\rfloor},\lambda_3\Hplus\lambda_4\right)$$

and hence  $c(\varrho_{m},\varrho_{m},\lambda_1 \Hplus\lambda_2 \Hplus\lambda_3 \Hplus\lambda_4)$.

It now suffices to take a $U_n$-typical partition $\lambda$ and show that it can be written as
$$\lambda=\lambda_1\Hplus\lambda_2\Hplus\lambda_3\Hplus\lambda_4,$$
where $\lambda_1,$ $ \lambda_2,$ $\lambda_3,$ $\lambda_4$ are dominance-comparable to the correspondingly sized staircases. 

We will do so by partitioning the columns into four sets corresponding to $\lambda_1,\cdots,\lambda_4$ such that each is dominance-comparable to $\varrho_{\lfloor \frac{m+1}{2}\rfloor}$ or $\varrho_{\lfloor \frac{m-1}{2}\rfloor}$. We give an algorithm that works for almost all partitions. On a macroscopic scale, we essentially will split the limit shape into 4 pieces of equal area by cutting it up vertically, as depicted below. We will then use some columns of length 1 to ensure that each $\lambda_i$ has exactly the correct total size, without interfering significantly with their overall shapes. We will denote by $\mu_1,\dots,\mu_4$ the 4 straightened continuous Young diagrams formed from the limit shape in this way. A figure below depicts the decomposition.

\begin{figure}[h]
\label{fig:ColumnParts}
\caption{Decomposition of the Uniform Limit Shape by Column Size}
\begin{center}\includegraphics{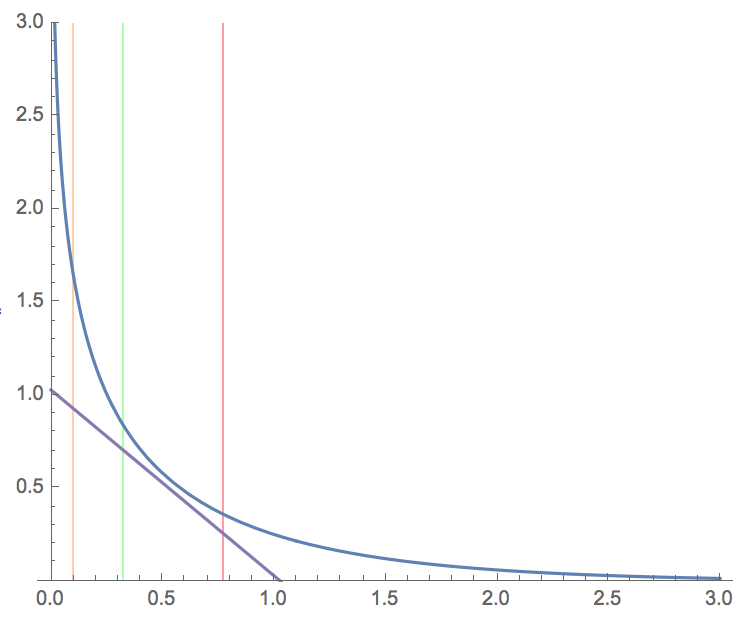}\end{center}
\end{figure}

Here the purple diagonal line represents the rescaled staircase partition $\varrho_{{\lfloor \frac{m}{2}\rfloor}}$, shifted right so that its corner is at the bottom of the green line (the line separating $\mu_2$ from $\mu_3$. We now briefly explain why the hypotheses of Lemma~\ref{lem:shapedomin} hold for each piece against a $\frac{1}{4}$-area $\varrho_{\infty}$. The exact equation for the curve of the limit shape boundary is $$e^{-\frac{\pi}{\sqrt{6}}x}+e^{-\frac{\pi}{\sqrt{6}}y}=1.$$

A convexity argument shows that on this curve, $x+y$ is minimized at $x=y=\frac{\sqrt{6}}{\pi}\log2$. Therefore, to ensure that the purple line does not intersect the curve, we need to verify that the x-coordinate for the green line is less than $(\frac{2\sqrt{6}}{\pi}\log2)-\frac{1}{\sqrt{2}}\approx 0.37$. Indeed, the green x-coordinate can easily be checked to be approximately $0.33$, which is smaller as desired. Therefore, the boundaries of $\varrho_{\infty}$ and $\mu_3$ intersect only once, including at the endpoints, so Lemma~\ref{lem:shapedomin} ensures dominance comparability. It is clear that dominance of $\mu_3$ implies the same for $\mu_2, \mu_1$. For $\mu_4$, it is easy to check that the red line meets the blue boundary curve at approximately $(0.78,0.36)$. Because this height is less than $\frac{1}{\sqrt{2}}\approx 0.71$, $\varrho_{\infty}$ has larger height than $\mu_4$. So it remains to check that the boundary curves of $\mu_4,\varrho_{\infty}$ intersect only at most 1 time. This is true because the boundary of $\mu_4$ consists only of points $(x,y)$ with $x>y$, and it is easily seen that the derivative of the boundary curve lies in the interval $(-1,0)$ on this region, so it cannot intersect a line of slope -1 twice. 

We will now decompose $\lambda$ by greedily partitioning the columns into four sets corresponding to $\lambda_1,\cdots,\lambda_4$ as follows. We fix some small $\varepsilon>0$ to be chosen later, and freely take $n$ to be large. Order the column lengths ${\lambda}_1^t\geq {\lambda}_2^t\geq\cdots$, and let $n_1$ be the largest integer such that $\sum_{i=1}^{n_1}{\lambda}_i^t\leq \binom{k+1}{2}$, and assign to $\lambda_1$ the columns ${\lambda}_1^t \dots {\lambda}_{n_1}^t$. By Lemma $\ref{lem:shapearea}$, the rescaled position of column $n_1$ is, with high probability, very close to the first vertical line in the diagram. Because it is also true w.h.p. that $\lambda$ is close to the limit shape in the metric on $\mathcal{SY}$, it is true w.h.p that no 2 subsequent adjacent columns $\lambda_m^t, \lambda_{m+1}^t$ $(m>n_1)$ differ by more than $\varepsilon \sqrt{n}$. This is simply because to the right of the orange line, the limit shape graph is uniformly continuous. 

Therefore, w.h.p. we may add one more column $\lambda_m^t$ to $\lambda_1$ so that $0 \leq \binom{k+1}{2}-|\lambda_1| \leq \varepsilon \sqrt{n}$, just by taking the largest column size $\lambda_m^t$ preserving the property $0 \leq \binom{k+1}{2}-|\lambda_1|$. 

We similarly form $\lambda_2,\lambda_3$ by greedily adding columns, and then (if necessary) adding one more column to give total size in the interval  $[\binom{k+1}{2}-i\varepsilon\sqrt{n},\binom{k+1}{2}]$ (We have a factor of $i \in \{2,3\}$ because the fact that we used out-of-place columns for previous $\lambda_j$ could increase the gap-size between available columns for subsequent $\lambda_i$.) We then greedily fill $\lambda_4$ with the remaining columns which are not of size 1. Note that $n_2$ and $n_3$ correspond almost exactly to the vertical boundaries between $\mu_i$, again by Lemma~\ref{lem:shapearea}. 

We now use the pieces of size 1 to smooth the exact sizes of $\lambda_i$. This is possible because of the following proposition.

\begin{prop}[{\cite[Thm 2.1]{Unif1s}}]
Let $X_1(\lambda)$ be the number of columns of size $1$ contained in a partition $\lambda$. For each $v\geq 0$, if $\lambda$ is a uniformly random partition of $n$,
$$\lim_{n\rightarrow\infty}P\left(\frac{\pi}{\sqrt{6n}}X_1\leq v\right)=1-e^{-v}.$$
\end{prop}

Because of this proposition, by picking $\varepsilon$ small, we have w.h.p. at least $6\varepsilon\sqrt{n}$ columns of size 1, enough to distribute to $\lambda_1,\lambda_2,\lambda_3$ in order to give them the correct sizes. We checked earlier that the condition for Lemma~\ref{lem:shapedomin} was satisfied for each $\mu_i$, so by also picking $\varepsilon$ small enough so as to use Lemma~\ref{lem:shapedomin}, we also have that dominance is preserved when we do this, so $\lambda_1,\lambda_2,\lambda_3$ are taken care of.

We now just add all remaining parts of size $1$ to $\lambda_4$. Each other $\lambda_i$ is the correct total size, so having used all the columns of $\lambda$, we see that $\lambda_4$ is also the correct size. Because $\mu_4$ dominates the infinite staircase $\varrho_{\infty}$, by Lemma~\ref{lem:shapedomin} again, with $\varepsilon$ sufficiently small the dominance will be preserved. (Note that because $\lambda_4$ dominates the corresponding staircase instead of being dominated as was the case for the other $\lambda_i$, we don't need to worry about adding too many singleton columns to $\lambda_4$.)

\end{proof}

\subsection{The Plancherel Case}

\begin{repthm}{thm:planch}
Let $m\geq 1$. Then $\varrho_m^{\otimes 2}$ contains almost all partitions of $\binom{m+1}{2}$ with respect to the Plancherel measure. 

\end{repthm}

\begin{proof}[Proof of Theorem~\ref{thm:planch}]

For this proof, we assume the following technical result, which will be proven in a later section.

\begin{defn} Let $\beta$ be a positive real number. A partition $\lambda$ of $n$ is called \emph{$\beta$-sum-flexible} if it satisfies the following property: when its column lengths are sorted in increasing order $a_1 \leq \dots a_m$, we have $a_1=1$, and for all $1\leq k \leq m$ we have $a_k \leq \left\lceil\beta \sum\limits_{j=1}^{k} a_j \right\rceil$.
\end{defn}

\begin{thm}
For any $\beta>0$, let $P(n,\beta)$ denote the probability that a (Plancherel) random partition of $n$ is $\beta$-sum flexible. Then we have
\[\lim_{n\rightarrow\infty} P(n,\beta) = 1\] for all $\beta$.
\label{thm:bess}
\end{thm}

To show Theorem~\ref{thm:planch}, we use the same summation identities for staircase partitions as above. Again, for simplicity, we discuss only the $n=2k$ case, the other case being nearly identical. 

Similarly to the uniform case, we will take a $M_n$-typical partition $\lambda$ and write it as $$\lambda=\lambda_1\Hplus\lambda_2\Hplus\lambda_3\Hplus\lambda_4,$$
where $\lambda_1,$ $ \lambda_2,$ $\lambda_3$ are dominance-comparable to $\varrho_k$ and $\lambda_4$ is comparable to $\varrho_{k-1}$. Unlike in the above, we can make each piece roughly the same shape. We will do so by grouping most of the column lengths $\lambda_1'\geq \lambda_2' \geq \dots$ into sets of four consecutive sizes, and dividing them cyclically among $\lambda_i$. We will use Theorem~\ref{thm:bess} to distribute the smallest columns, in such a way that $|\lambda_1|=\lambda_2|=|\lambda_3|=\binom{k+1}{2}$ and $|\lambda_4|=\binom{k}{2}$. Similarly to the uniform case, define $\mu$ to be the Plancherel limit shape, but transformed by the linear map $(x,y)\rightarrow(\frac{x}{2},2y)$. 

We first argue that $\mu, \varrho_{\infty}$ satisfy the hypotheses of Lemma~\ref{lem:shapedomin}. We have $\ell(\varrho_{\infty})=h(\varrho_{\infty})=1, \ell(\mu)=\frac{1}{\sqrt{2}}, h(\mu)=2\sqrt{2}$, so it remains to check that their boundaries intersect only once off of the axes. In Russian coordinates, extend the boundary of $\mu$ along the axes to a convex function $f$ defined on all of $\mathbb{R}$ by setting $f(x)=|x|$ outside the boundary curve of $\mu$. Since $f$ is convex, it can only meet any line twice. The boundary of $\varrho_{\infty}$ is a line, and one intersection point is $(1,1)$. Therefore, the boundaries of $\mu,\varrho_{\infty}$ intersect at most once, so they do meet the conditions of Lemma~\ref{lem:shapedomin}. 

Fix a miniscule $\varepsilon > 0$. We order the column lengths $\lambda_1^t \geq \lambda_2^t \geq \dots $. Take all columns with size at least $\varepsilon n^{1/2}$ and add column $\lambda_i^t$ to partition $k$ if $i$ and $k$ are congruent modulo 4. Because each partition's total size corresponds essentially to a Riemann integral of the limit shape, each partition now has size $n(\frac{1}{4}-\delta+o(1))$, where $\delta$ is the area of the limit shape covered by columns are size less than $\epsilon$ after rescaling \cite{Baik}.
This means that each partial partition is smaller than the size of the corresponding $\lambda_i$, which all have size $n(\frac{1}{4}-o(1))$. By Theorem~\ref{thm:bess}, we may distribute the remaining columns to give each partition the correct overall size. We claim that these remaining small parts occupy an infinitesimal area in the rescaled diagram: they fit into a square of arbitrary small side length, for $\varepsilon$ sufficiently small. To verify this, we need that the longest row of $\lambda$ is of size $(2+o(1))\sqrt{n}$ with high probability, i.e. the same length as predicted by the length of the limit shape. This is indeed true, see \cite{Baik}.

This implies that the resulting shapes are very close to $\mu$ in the metric on $\mathcal{SY}$, meaning that, assuming we picked a sufficiently small $\varepsilon$ value originally, we may apply Lemma~\ref{lem:shapedomin} to conclude that we have dominance.

In summary, we have decomposed a generic large partition as a horizontal sum of 4 almost-equally-shaped partitions, each comparable to their size of staircase via the dominance ordering. Hence, this generic large partition is a constituent of $\varrho_m^{\otimes 2}$, by the semigroup property, and so we have established Theorem~\ref{thm:planch}.

\end{proof}

\section{Overview of the Deterministic Approach}

In this section, we consider a different weakening of the tensor square conjecture:

\begin{ques}
What is the smallest integer $f(m)$ such that $\varrho_m^{\otimes f(m)}$ contains every partition of $n=\frac{m(m+1)}{2}$?

\end{ques}

The conjecture is that $f(m)=2$; we seek here to find any good upper bound on $f(m)$. It may seem natural to attempt to use the semigroup property directly for this question; we can in fact define higher-length Kronecker coefficients via constituency in longer tensor products, and the semigroup property still holds for these longer sequences in the same way (see Appendix~\ref{subsec:rectcube}). However, we take a different approach, by instead allowing for additional factors of the standard representation $\tau_n$, and then replacing these factors with factors of $\varrho_m$ to answer the above question.

\subsection{Blockwise Distances}

The square-moving interpretation of tensoring a representation with $\tau_n$ motivates the following definition:

\begin{defn}
The \textit{blockwise distance} $\Delta(\lambda,\mu)$ between two partitions $\lambda,\mu$ of $n$ is the smallest number of single blocks that need to be moved to form $\lambda$ from $\mu$.
\end{defn}

Equivalently, the blockwise distance is the smallest $\Delta$ such that $\lambda\otimes \tau_n^{\otimes \Delta}$ contains $\mu$. To motivate our approach in this section, consider the graph of Young diagrams of size $n$ in which $\lambda$ and $\mu$ are adjacent if and only if they differ by the movement of exactly 1 block. Then $\lambda\otimes\tau_n$ is the formal sum of $\lambda$ (possibly many times) and all of its neighbors (once each). If we are only concerned (as we are) with which representations appear at all, we see that tensoring with $\tau_n$ corresponds simply to ``spreading out" along this graph. This leads us to another weakening of the tensor square conjecture:

\begin{defn}

Define $H(m)$ to be the minimum non-negative integer $\ell$ such that $\varrho_m^{\otimes 2}\otimes \tau_n^{\otimes \ell}$ contains all partitions of $n$ as constituents, where $n=\frac{m(m+1)}{2}$.

\end{defn}

\begin{ques}

How quickly does $H(m)$ grow with $m$?

\end{ques}

If $H(m)$ is small, then partitions contained in $\varrho_m^{\otimes 2}$ are ``dense" in the graph described above. In this section we find an upper bound for the number of standard representation factors required, and then translate this into a bound $f(m)$. We also define a slightly more general distance which allows differently sized partitions to be compared.

\begin{defn}
The \textit{generalized blockwise distance} $\Delta(\lambda,\mu)$ between two partitions $\lambda\vdash n_1,\mu\vdash n_2$ is the smallest number of single blocks that need to be added, removed, and/or moved to form $\lambda$ from $\mu$.
\end{defn}

\begin{rem} When $\lambda,\mu$ are partitions of the same $n$, the generalized distance $\Delta({\lambda,\mu})$ is the same as the regular blockwise distance.
\end{rem}

First we verify that horizontally adding partitions interacts simply with blockwise distances.

\begin{prop}
\label{prop:distbound}
If $\lambda_1,\mu_1,\lambda_2,\mu_2$ are integer partitions then $\Delta(\lambda_1,\lambda_2)+\Delta(\mu_1,\mu_2)\geq \Delta(\lambda_1\Hplus\mu_1,\lambda_2\Hplus\mu_2)$.
\end{prop}
\begin{proof}
Consider a sequence of $d_1$ operations on the columns of $\lambda_1$ that can be done to result in $\lambda_2$, and consider the analogous sequence for transforming $\mu_1$ to $\mu_2$. Since horizontal addition simply combines the multisets of column sizes, it suffices to show that the same two sequences of moves can be performed separately on the sets of columns of $\lambda_1\Hplus\mu_1$ corresponding to $\lambda_1$ and $\mu_1$, respectively, to achieve the same transformation as before for each of the two parts.

This follows because the basic operations, when performed on the columns of $\lambda_1\Hplus\mu_1$ corresponding to $\lambda_1$, do not affect those corresponding to $\mu_1$, and vice versa; when the relative sizes of two columns would change, we can simply reorder the columns without affecting the result.
\end{proof}

Note that $\Delta({\lambda,\mu})\leq n-1$ for all $\lambda,\mu\vdash n$. As a demonstration of what using standard representations can give us, we give the following pair of weak but instructive results:

\begin{prop}
\label{prop:stdInStair}
Let $n=\frac{m(m+1)}{2}$. If $\lambda\vdash n$ is such that $\lambda$ is contained in $\tau_n^{\otimes m}$, then $\lambda$ is contained in $\varrho_m^{\otimes 2}$.
\end{prop}
\begin{proof}
The irreducible components of $\tau_n^{\otimes m}$ are precisely those that correspond to partitions $\lambda$ of $n$ whose blockwise distance from $1_n$ is at most $m$. If this blockwise distance is at most $m-1$, then $\lambda$, minus its top row, is contained within $\varrho_{m-1}$, and so $\lambda$ is comparable to $\varrho_m$ in dominance order. Thus we get $c(\varrho_m,\varrho_m,\lambda)$.

If the distance is $m$, the only way for $\lambda$ not to be comparable to $\varrho_m$ is if $\lambda=(n-m,1,1,\cdots,1)$. But in this case, we have $c(\varrho_m,\varrho_m,\lambda)$ anyway by the semigroup property, because $\lambda$ is a hook (\cite{IkenDom}).
\end{proof}

\begin{prop}
\label{prop:trivSqrtStair}
Let $n=\frac{m(m+1)}{2}$. All $\lambda\vdash n$ are contained in $\varrho_m^{\otimes 2\lceil\frac{m+1}{2}\rceil}$.
\end{prop}
\begin{proof}
All $\lambda\vdash n$ are contained in $\tau_n^{\otimes n-1}$ and thus in $\tau_n^{\otimes m\lceil\frac{m+1}{2}\rceil}$. So for each $\lambda$, there exist irreducible representations  $\mu_1,\cdots,\mu_{\lceil\frac{m+1}{2}\rceil}$, each contained in $\tau_n^{\otimes m}$, whose tensor product contains $\lambda$. Thus $\lambda$ is in the tensor product of $\lceil\frac{m+1}{2}\rceil$ irreducible representations each contained in $\varrho_m^{\otimes 2}$, and thus $\lambda$ is contained in $\varrho_m^{\otimes 2\lceil\frac{m+1}{2}\rceil}$
\end{proof}
\begin{rem} This already shows that $f(n)=O(\sqrt{n})$. As we will soon see, we can do much better.
\end{rem}

\subsection{Staircase Sum Identities}
As seen previously, staircases can be added to form larger staircases. The most straightforward way follows, and is a generalization of the identities used in the previous section (which we recover when $k=2$). We use the symbols $\hsum,\vsum$ for iterated applications of $\Hplus,\Vplus$.

\begin{prop} For any $n,k$ we have:
$$\varrho_n=\vsum_{j=0}^{k-1} \left(\hsum_{i=0}^{k-1} \left(\varrho_{\left\lfloor \frac{n+i-j}{k} \right\rfloor}\right)\right).$$

\label{thm:stairgrid}
\end{prop}

\begin{proof}

Recall Hermite's identity, which states that for any $x\in\mathbb R$ we have 

$$\sum_{i=0}^{k-1}\left\lfloor x+\frac{i}{k}\right\rfloor=\lfloor kx\rfloor.$$

Thus, the $j$th term in the vertical sum is a partition whose largest row has size

$$\sum_{i=0}^{k-1}\left\lfloor \frac{n+i-j}{k}\right\rfloor = n-j.$$ 

Since the $k$ summands are all staircases with length differing by at most one, the other rows in the sum have size $n-j-k, n-j-2k,\cdots$. Thus the vertical sum creates a partition with row sizes $n,n-1,\cdots,1$, which is what we want.
\end{proof}

For example, for $k=2$ we recover the identities we used in the previous section:
$$
\begin{aligned}
(\varrho_{k} \Hplus \varrho_{k-1}) \Vplus (\varrho_k \Hplus \varrho_k) = \varrho_{2k},\\ (\varrho_{k+1} \Hplus \varrho_{k}) \Vplus (\varrho_k \Hplus \varrho_k) = \varrho_{2k+1}.
\end{aligned}$$

As was visually demonstrated before, these identities may be understood as beginning with a square grid of staircases $\varrho_{\lfloor\frac{n+i-j}{k}\rfloor}$ and horizontally adding the rows, then vertically adding the resulting sums. Of course, we could also first add vertically by column, and then horizontally by row. 

This square grid of staircase partitions can also be added in an entirely different order to produce the same result. Instead of adding all the staircases in each row together, we can first add an $i$-by-$i$ square, and then add a length-$i$ column and a length $i+1$ row to get an $(i+1)$-by-$(i+1)$ square. However, some slightly messy rearrangement of the pieces is needed, and the following proposition is the result.

\begin{prop} 

For any $n,k$ we have:
\begin{enumerate}
\item $$\left(\varrho_n\Hplus ((k-1)\subV \varrho_{\lfloor\frac{n}{k-1}\rfloor})\right)\Vplus \left(\vsum_{i=0}^{k-1}\left(\varrho_{\left\lfloor\frac{n+\lfloor\frac{n}{k-1}\rfloor-k+1+i}{k}\right\rfloor}\right)\right)=\varrho_{n+\lfloor \frac{n}{k-1}\rfloor},$$

\item $$\left(\varrho_n\Hplus ((k-1)\subV \varrho_{\lfloor\frac{n}{k-1}\rfloor})\right)\Vplus \left(\vsum_{i=0}^{k-1}\left(\varrho_{\left\lfloor\frac{n+\lfloor\frac{n}{k-1}\rfloor+1+i}{k}\right\rfloor}\right)\right)=\varrho_{n+\lfloor \frac{n}{k-1}\rfloor+1}.$$

\end{enumerate}

\label{thm:layergrid}
\end{prop}

The proof of this proposition will be given in Section~\ref{subsec:stairRigor}.

Below is a visual illustration of case 1 with $k=4, n=7$, where the pieces combine to make a $\varrho_9$.

\begin{figure}[h]
\label{fig:layersadd}
\caption{Layer Decomposition}
{\centering \includegraphics[scale=0.6]{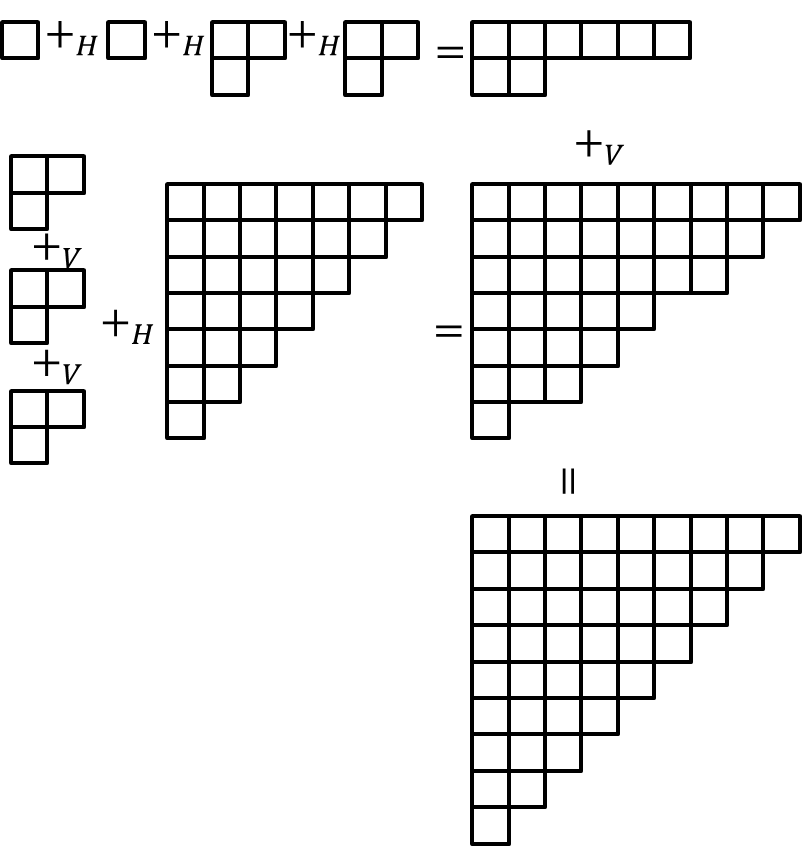}\par }
\end{figure}

Using the above, we can construct a decomposition of any large staircase into a staircase of about $(\frac{i}{k})^2$ times its area and $k^2-i^2$ staircases of about $\frac{1}{k^2}$ times its area. This may be done by repeatedly applying Theorem~\ref{thm:layergrid}.

\subsection{Overview of Proof of Theorem~\ref{thm:4thpower}}

Here is our main result, stated again.

\begin{repthm}{thm:4thpower}
For sufficiently large $n$, there exists $\lambda\vdash n$ such that $\lambda^{\otimes 4}$ contains all partitions of $n$.
\end{repthm}

For this overview, we will focus on the case of a triangular number $n=\frac{m(m+1)}{2}$, using for $\lambda$ the staircase $\varrho_m$. The general case will augment most steps simply by adding the appropriate number of extra blocks in the form of a trivial representation. The detailed proof can be found in Section~\ref{sec:detRes}.

The primary ingredient is the following theorem.

\begin{thm}
\label{thm:linearH}
For all $m$, $\varrho_m^{\otimes 2}\otimes \tau^{\otimes O(m)}$ contains all partitions of $n=\binom{m+1}{2}$.
\end{thm}

Equivalently, we will show that every partition is in $\varrho_m^{\otimes 2}\otimes \tau^{\otimes \lfloor c\sqrt{n}\rfloor}$ for an absolute constant $c$.

We defined earlier $H(m)$ to be the minimum number of tensor factors of $\tau_n$ needed so that $\varrho_m^{\otimes 2}\otimes \tau_n^{\otimes H(m)}$ contains every partition, so Theorem~\ref{thm:linearH} states that $H(m)=O(m)$. For convenience, we define a maximal function for the function $H$.
\begin{defn}

For $m\geq 0$ let $M(m)=\max\{H(1),H(2),\dots,H(m)\}.$

\end{defn}

It is clear that $M(m)=O(m) \iff H(m)=O(m)$, and so we will show the former. 

\begin{proof}[Proof Outline for Theorem~\ref{thm:linearH}]

We use Proposition~\ref{thm:layergrid} to decompose $\varrho_m$ as

$$\varrho_m=(\varrho_x\Hplus (3\subV \varrho_y))\Vplus\left(\vsum_{i=0}^{3} \varrho_{z_i}\right),$$

where up to $O(1)$ error, we have $x\approx \frac{3m}{4},$ $  y\approx z_i\approx \frac{m}{4}$. 

Suppose that we have eight partitions $\lambda_0\dots \lambda_{7}$ each contained in the tensor square of one of the above staircases. Then repeatedly combining the partitions and using the semigroup property gives us
$$c\left(\varrho_m,\varrho_m,\hsum_{i=0}^{7}(\lambda_i)\right).$$
\\

So, we only need to show how to decompose an arbitrary partition $\mu$ as a horizontal sum of many smaller partitions comparable to staircases, using standard representations to move a few blocks around as necessary. We will plan for all but two of these smaller partitions to be dominance-comparable to the staircase. These will then be the only two smaller partitions that need additional tweaking with standard representations.\\

The idea is the following: we cut off a large chunk $A$ of $\mu$, of about $\frac{7}{16}$ its size. The remainder $\lambda_0$ (consisting of the larger end of the columns) is modified recursively to be contained in the tensor square of the staircase of size about $\frac{9}{16}n$; $A$ itself is broken into $7$ pieces which can be modified to be suitable $\lambda_i$. This breaking is done contiguously, so that columns in a piece are either all at least as tall or all at least as short as columns in a given other piece. If all columns in one of these pieces are either very tall or very short, we can conclude dominance comparability with the staircase. By further breaking down the smaller pieces, we can ensure using this criterion that most of the small parts in our decomposition are dominance comparable to the corresponding staircases. The subadditivity of blockwise distance allows us to derive a recursion for $M(m)$ from this decomposition, and we do not have recursive terms corresponding to most of these smaller pieces. The end result is the following recurrence.

$$M(m)\leq M\left(\frac{3m}{4}+O(1)\right)+M\left(\frac{m}{8}+O(1)\right)+cm+O(1).$$

Solving this recurrence by strong induction gives $$M(m)\leq Cm$$ for some constant $C$ (see Section~\ref{subsec:close2all} for an explicit $C$ which works for large $n$). This proves that $M(m)=O(m)$, and hence Theorem~\ref{thm:linearH}.

\end{proof}

To conclude Theorem~\ref{thm:4thpower} it remains to show that for any $c$, for large enough $n$, any irreducible representation in $\tau^{\otimes \lfloor c\sqrt{n}\rfloor}$ is contained in $\varrho_m^{\otimes 2}$. In the next section, we explain our results in more detail and fill in this last step to complete the proof of Theorem~\ref{thm:4thpower} even in non-triangular number cases. \\

\section{Detailed Proof of Theorem~\ref{thm:4thpower}}
\label{sec:detRes}
We follow the same plan as in the proof overview, giving more detail and showing the extensions to non-triangular values of $n$. First, we rigorize some of our earlier work with staircase sum identities.

\subsection{Staircase Sum Identities, Revisited}
\label{subsec:stairRigor}

Here is a result we used in the overview, which will now be proven. 

\begin{repprop}{thm:layergrid}

For any $n,k$ we have:
\begin{enumerate}
\item $$\left(\varrho_n\Hplus ((k-1)\subV \varrho_{\lfloor\frac{n}{k-1}\rfloor})\right)\Vplus \left(\vsum_{i=0}^{k-1}\left(\varrho_{\left\lfloor\frac{n+\lfloor\frac{n}{k-1}\rfloor-k+1+i}{k}\right\rfloor}\right)\right)=\varrho_{n+\lfloor \frac{n}{k-1}\rfloor},$$

\item $$\left(\varrho_n\Hplus ((k-1)\subV \varrho_{\lfloor\frac{n}{k-1}\rfloor})\right)\Vplus \left(\vsum_{i=0}^{k-1}\left(\varrho_{\left\lfloor\frac{n+\lfloor\frac{n}{k-1}\rfloor+1+i}{k}\right\rfloor}\right)\right)=\varrho_{n+\lfloor \frac{n}{k-1}\rfloor+1}.$$

\end{enumerate}

\end{repprop}

\begin{proof}

We simply check each equality by evaluating the left sides. We begin with 1. We have $\varrho_n=(n,n-1,\dots 1)$. Because vertical addition is equivalent to taking the disjoint union of the row-multisets, we have $(k-1)\subV \varrho_{\lfloor\frac{n}{k-1}\rfloor}=(\lfloor\frac{n}{k-1}\rfloor,\lfloor\frac{n}{k-1}\rfloor,\dots, 1)$ where each distinct value is repeated $k-1$ times. Therefore,

$$\left(\varrho_n\Hplus ((k-1)\subV \varrho_{\lfloor\frac{n}{k-1}\rfloor})\right)=
\left(n+\left\lfloor\frac{n}{k-1}\right\rfloor, \dots, n+\left\lfloor\frac{n}{k-1}\right\rfloor-k+2,n+\left\lfloor\frac{n}{k-1}\right\rfloor-k,\dots\right)$$ 
where the omitted row lengths are precisely the values $n+\lfloor\frac{n}{k-1}\rfloor+1-jk$ for positive integral $j$. Again using the fact the vertical addition is simply disjoint union of row lengths, it suffices to check that

$$\left(\vsum_{i=0}^{k-1}\left(\varrho_{\left\lfloor\frac{n+\lfloor\frac{n}{k-1}\rfloor-k+1+i}{k}\right\rfloor}\right)\right)$$
consists of precisely these row lengths. That the largest row lengths match follows from Hermite's identity. Because the $k$ numbers ${\left\lfloor\frac{n+\lfloor\frac{n}{k-1}\rfloor-k+1+i}{k}\right\rfloor}$ differ pairwise by at most 1, the row lengths of
$$\left(\vsum_{i=0}^{k-1}\left(\varrho_{\left\lfloor\frac{n+\lfloor\frac{n}{k-1}\rfloor-k+1+i}{k}\right\rfloor}\right)\right)$$
will decrease by $k$ until reaching $0$, which is exactly the correct behavior.

The proof of 2. is identical, except now the omitted row lengths are those of the form $n+\left\lfloor\frac{n}{k-1}\right\rfloor+1-jk$ for \textit{non-negative} integers $j$. 

\end{proof}

It is easy to see that the function $f(n)=n+\left\lfloor\frac{n}{k-1}\right\rfloor$ attains all integer values except those congruent to $-1 \mod k$. Such values are attained by $f(n)+1$, so Proposition~\ref{thm:layergrid} suffices to break up any staircase into smaller pieces.

We now use the above proposition to construct a decomposition of any large staircase into a staircase of about $(\frac{i}{k})^2$ times its area and $k^2-i^2$ staircases of about $\frac{1}{k^2}$ its area, for any fixed $i<k$.

\begin{defn}

Define a \emph{$k$-layer decomposition} of $\varrho_m$ as a decomposition of $\varrho_m$ into $2k$ staircases which have sizes given by the left hand side of equation 1. of Proposition~\ref{thm:layergrid}, where $m=n+\left\lfloor\frac{n}{k-1}\right\rfloor$, and which sum to $\varrho_m$ in the way indicated by that equation. The piece $\varrho_n$ in this decomposition is called the \emph{core}.

For $1\leq i\leq k-1$, define a \emph{$(k,i)$-layer decomposition} of $\varrho_m$ recursively as follows:

A $(k,k-1)$-layer decomposition of $\varrho_m$ is a $k$-layer decomposition of $\varrho_m$.

For $i<k-1$, a $(k,i)$-layer decomposition of $\varrho_m$ is the result of taking a $(k,i+1)$-layer decomposition of $\varrho_m$, and further decomposing its core through a $(k-1,i)$-layer decomposition.
\end{defn}

Thus, a $(k,i)$-layer decomposition is a decomposition of $\varrho_m$ into $k^2-i^2+1$ smaller staircases. We extend the definition of the core to these decompositions and introduce a related term for the remaining pieces.

\begin{defn}

We call the large part of a $(k,i)$-layer decomposition of $\varrho_m$ the \textit{core} and the other $k^2-i^2$ parts the \textit{flakes}.
\end{defn}

It is clear that, for fixed $(i, k)$, the flakes formed differ by only $O(1)$ in length.

In fact, as long as $2i\leq k$, by using Proposition $\ref{thm:layergrid}$ intelligently we can ensure that they differ by at most 1.

\begin{defn}

We call a $(k,i)$-layer decomposition of $\varrho_m$ into parts where all flakes pairwise differ in length by at most $1$ a \textit{smooth layer decomposition}.

\end{defn}

\begin{prop}
\label{thm:layersmooth}

For any $(m,k,i)$ with $2i\leq k$, there is a smooth $(k,i)$-layer decomposition of $\varrho_m$.
\end{prop}

\begin{proof}
We first examine based on the value of $j$ to see which flake lengths can arise in a $(k,1)$ decomposition.

If $0\leq j\leq k-2$ then we may use part 1. of Proposition~\ref{thm:layergrid}. The core length is $n=t(k-1)+j$, which means the flake lengths are $\left\lfloor\frac{n}{k-1}\right\rfloor=t$ and ${\left\lfloor\frac{n+\left\lfloor\frac{n}{k-1}\right\rfloor-k+1+i}{k}\right\rfloor}=\left\lfloor\frac{t(k-1)+j+t-k+1+i}{k}\right\rfloor=\left\lfloor\frac{(t-1)(k)+j+1+i}{k}\right\rfloor$. Because $i\leq k-1$ we have $(t-1)k+j+1+i<(t+1)k$, so these numbers range over the set $\{t-1,t\}$.

If $1\leq j\leq k-1$ then we may use part 2. of Proposition~\ref{thm:layergrid}. The core length is $n=t(k-1)+j-1$, which makes the flake lengths $\left\lfloor\frac{n}{k-1}\right\rfloor=t$ and  ${\left\lfloor\frac{n+\left\lfloor\frac{n}{k-1}\right\rfloor+1+i}{k}\right\rfloor}=\left\lfloor\frac{t(k-1)+j-1+t+1+i}{k}\right\rfloor=t+\left\lfloor\frac{j+i}{k}\right\rfloor$. This clearly ranges over $\{t,t+1\}$.

Now, assume first that $j<k-i$. Then we may use 1. of Proposition~\ref{thm:layergrid} $i$ times, having after $\ell$ iterations a core of length $t(k-\ell)+j$ and all flakes of lengths in $\{t-1,t\}$. Because $j<k-i$, we always have $j\leq k-\ell-2$ for $\ell \leq k-1$, so we can complete the $i$ iterations in this way.

If $j\geq k-i$, then as $2i\leq k$ we have $j\geq i$. Then we may instead use only part 2. of the proposition, which after $\ell$ iterations leaves a core of length $t(k-\ell)+j-\ell$. Because $j\geq i$, this is strictly greater than $t(k-\ell)$ for all $\ell < i$, so we can again run this process to completion, with all flakes of length in $\{t,t+1\}$. In either case, we are done.

\end{proof}

\subsection{Proof of Theorem \ref{thm:linearH}}
\label{subsec:close2all}

We now prove Theorem~\ref{thm:linearH}, restated with an explicit asymptotic constant.

\begin{repthm}{thm:linearH}[with explicit constant]
For all $m$, $\varrho_m^{\otimes 2}\otimes \tau^{\otimes (1184m+O(1))}$ contains all partitions of $n=\binom{m+1}{2}$.
\end{repthm}

\begin{proof}

We use Proposition~\ref{thm:layergrid} to write $\varrho_m$ as a $(k,1)$ decomposition

$$\varrho_m=(\varrho_x\Hplus ((k-1)\subV \varrho_y))\Vplus\left(\vsum_{i=0}^{k-1} \varrho_{z_i}\right),$$
where up to $O(1)$ error, we have $x\approx \frac{(k-1)m}{k},$ $  y\approx z_i\approx \frac{m}{k}$. 

Now, suppose that we have partitions $\lambda_0\dots \lambda_{2k-1}$ such that

$$ c(\varrho_x,\varrho_x,\lambda_0),$$ 

$$c(\varrho_y,\varrho_y,\lambda_j)$$
for $1 \leq j \leq k-1$ and 
$$c(\varrho_{z_i},\varrho_{z_i},\lambda_{k+i})$$
for $0\leq i \leq k-1$.

Then repeated application of the semigroup property implies 

$$c\left((k-1)\subV \varrho_y,(k-1)\subV \varrho_y,\hsum_{i=1}^{k-1} (\lambda_i)\right),\hspace{1cm} 
c\left(\vsum_{j=0}^{k-1} (\varrho_{z_j}),\vsum_{j=0}^{k-1} (\varrho_{z_j}),\hsum_{i=k}^{2k-1} (\lambda_{i})\right),$$

which in turn imply

$$c\left(\varrho_x\Hplus ((k-1)\subV \varrho_y),(\varrho_x\Hplus ((k-1)\subV \varrho_y),\hsum_{i=0}^{k-1} (\lambda_i))\right) $$

and so

$$c\left(\varrho_x\Hplus ((k-1)\subV \varrho_y)\Vplus\left(\vsum_{i=0}^{k-1} \varrho_{z_i}\right),\varrho_x\Hplus ((k-1)\subV \varrho_y)\Vplus\left(\vsum_{i=0}^{k-1} \varrho_{z_i}\right),\hsum_{i=0}^{2k-1}(\lambda_i)\right)$$

$$\implies c\left(\varrho_m,\varrho_m,\hsum_{i=0}^{2k-1}(\lambda_i)\right).$$

Thus, we only need to show how to decompose an arbitrary partition as a horizontal sum of many smaller partitions.

\begin{lem}
\label{lem:domheights}
Suppose $\lambda\vdash n$, where $n=\frac{k(k+1)}{2}$, and the columns of $\lambda$ have sizes $c_1\geq c_2\geq \cdots \geq c_l$. Then if either $c_l\geq k$ or $c_1\leq \left\lfloor\frac{k}{2}\right\rfloor+1$, we have $c(\varrho_k,\varrho_k,\lambda)$.
\end{lem}
\begin{proof}
First consider the case where $c_l\geq k$. Then, for $1\leq j \leq l$, we have $\sum_{i=1}^j c_i\geq kj\geq k+(k-1)+\cdots+(k-(j-1))$, while for $j>l$, $\sum_{i=1}^j c_i\geq n$. So, $\lambda$ is clearly dominance-comparable to the staircase.

Now consider the case where $c_1\leq \left\lfloor\frac{k}{2}\right\rfloor+1$. Now for any $j<k$, the leftmost $j$ columns of $\varrho_k$ have average size $\frac{k+(k-j+1)}{2}\geq \frac{k+2}{2}$, which is at least the average size of the first $j$ columns of $\lambda$. For $j\geq k$, the sizes of the leftmost $j$ columns of $\varrho_k$ sum to $n$. So, $\lambda$ and $\varrho_k$ are again comparable in the dominance order.

Thus, in either case, $c(\varrho_k,\varrho_k,\lambda)$ follows by Theorem~\ref{thm:dominance}.
\end{proof}

We now introduce a lemma which allows us to break up a partition $\mu$ with bounded height into partitions contained in staircase tensor squares. The reader may wish to note that only the first part of the lemma is needed to conclude that $O(\sqrt{n})$ standard representations suffice to give every representation, by taking (e.g.) $(k,i)=(7,5)$ instead of $(k,i)=(4,3)$. However, using the second part results in a better constant factor. Recall that we defined $M(m)=\max\{H(1),H(2),\dots,H(m)\}$, where $H(m)$ was the minimal non-negative integer $\ell$ such that $\varrho_m^{\otimes 2}\otimes\tau_n^{\otimes \ell}$ contains all partitions of $n$ as constituents.

\begin{lem}
\label{lem:cuttail}

Consider a partition $\mu$ and staircases $\varrho_{s_1}, \dots , \varrho_{s_r}$ with $s_i \in \{b-1,b\}$ for all $i$. Assume that the height of the largest column $c_1$ satisfies $c_1\leq C$ and also $0\leq (\sum\limits_{i} |\varrho_{s_i}|)-|\mu|  \leq C$. Then there exists a partition $\widehat{\mu}$ such that  $c\left(\hsum_{i=1}^r \varrho_{s_i},\hsum_{i=1}^r \varrho_{s_i},\widehat{\mu}\right)$ and also the generalized blockwise distance $\Delta=\Delta(\widehat{\mu},\mu)$ between $\widehat{\mu}$ and $\mu$ satisfies

\begin{enumerate}
\item[(1)] $\Delta\leq (4r-3)C+M(b)$

and

\item[(2)] $\Delta\leq (4r+9)C+M(\lceil \frac{b}{2}\rceil)$

\end{enumerate}
\end{lem}

\begin{proof}

We break up $\mu$ by its column lengths $c_1\geq  c_2\dots \geq c_h$. We will construct partitions $\zeta_1\dots\zeta_{r}$ which horizontally sum to $\widehat{\mu} \approx \mu$. We begin by constructing partitions $\zeta_i^*$, and then adjusting each $\zeta_i^*$ to form $\zeta_i$. 

To form $\zeta_i^*$, we preliminarily use the greedy algorithm to assign to $\zeta_1^*$ as many columns from $\mu$ as possible subject to $|\zeta_1^*|\leq |\varrho_{s_1}|$ starting from the smallest column $c_h$, then continue similarly with $\zeta_2^*$ and so on through $\zeta_r^*$. Note that some columns of $\mu$ may be included in none of the $\zeta_i^*$. The assumptions $c_1\leq C$ and $0\leq \sum\limits_{i} |\varrho_{s_i}|-|\mu|  \leq C$ guarantee that for each $i$, we have $$|\zeta_i^*|\leq |\varrho_{s_i}| \leq |\zeta_i^*|+C.$$

Now we modify each $\zeta_i^*$ based on the height condition of Lemma~\ref{lem:domheights}. For each $i$ set $X_i$ to be the height of the smallest column of $\zeta_i$ and $Y_i$ to be the height of the largest. Clearly we have $X_1\leq Y_1\leq X_2 \dots \leq X_r \leq Y_r$. We leave alone those $\zeta_i^*$ with $Y_i \leq \lceil \frac{b}{2} \rceil$ - such $\zeta_i$ can easily be modified to satisfy dominance by Lemma~\ref{lem:domheights} as we will see later. 

We have some remaining set $S$ of $i$ such that $Y_i > \lceil \frac{b}{2} \rceil$. Now, because we have $X_{i+1}\geq Y_i$, at most 1 value of $i\in S$ satisfies $X_i \leq \lceil \frac{b}{2}\rceil$, namely, only the minimal value $i_0$ of $S$. Therefore, for all $i \in S\backslash\{i_0\}$, we have $X_i > \lceil \frac{b}{2}\rceil$. 

For each $i\in S\backslash \{i_0\}$ we break $\varrho_{s_i}$ into a sum of $4$ smaller staircases $\varrho_{s_{i,1}}\dots\varrho_{s_{i,4}}$  by setting $k=2$ in Proposition~\ref{thm:stairgrid}. We also use the same greedy algorithm as before to extract from $\zeta_i^*$ four partitions $\zeta^*_{i,1}\dots\zeta^*_{i,4}$ such that $|\varrho_{s_{i,j}}|-C\leq |\zeta^*_{i,j}|\leq|\varrho_{s_{i,j}}|$ and the $\zeta^*_{i,j}$ use distinct columns of $\zeta^*_i$. 

Because $s_i\leq b$, we have $s_{i,j} \leq \lceil \frac{b}{2} \rceil$. For all $i\in S\backslash \{i+0\}$, $X_i > \lceil \frac{b}{2}\rceil$, and the smallest column in any $\zeta^*_{i,j}$ is certainly at least as large as $X_i$. This means that the $\zeta^*_{i,j}$ now satisfy the other size condition in Lemma~\ref{lem:domheights}; by breaking down our partitions, we have preserved the ``tallness" of the $\zeta^*_i$ in the $\zeta^*_{i,j}$ while decreasing the total size, making them sufficiently ``relatively tall" to apply Lemma~\ref{lem:domheights}.

Now we finish part (1). We will add squares to the small staircases $\zeta^*_i$, $\zeta^*_{i,j}$ while preserving the respective height conditions of Lemma~\ref{lem:domheights}. Let $R$ be the set of partitions consisting of $\zeta_i$ for $i\not \in S$ and $\zeta_{i,j}$ for $i\in S\backslash\{i_0\}$. We know that for each partition in $R$, we can bring its total size to the size of the corresponding staircase $\varrho_{s_i}$ or $\varrho_{s_{i,j}}$ by adding at most $C$ squares. 

We consider $\mu$ as the disjoint union of column-length multi-sets of the partitions in $R$, $\zeta^*_{i_0}$, and the remaining leftover columns. We make block-moves to bring each partition $\zeta^*_i$ or $\zeta^*_{i,j}$ in $R$ to be of size $s_i$ or $s_{i,j}$, while preserving the respective condition from Lemma~\ref{lem:domheights}. Note that because we have $|\zeta^*_i|\leq |\varrho_{s_i}|$ and $|\zeta^*_{i,j}| \leq |\varrho_{s_{i,j}}|$, we can achieve the first goal with all block-moves being the addition of a new block.

We show that we can add new blocks freely without destroying our height conditions. For $i \not\in S$, we need to avoid increasing the height of $\zeta^*_i$, and we do this by adding additional blocks only as new columns of length $1$. The resulting modifications of $\zeta^*_i$ our our desired $\zeta^*_i$.  For $i\in S\backslash\{i_0\}$, we need to avoid decreasing the minimum height of the columns of $\zeta^*_{i,j}$, and we do this by adding blocks only as rows of length $1$, or equivalently increasing the longest column length. The resulting modifications of $\zeta^*_{i,j}$ are our $\zeta_{i,j}$.

We clearly have $|R|\leq 4r-4$. Recall that we are consider $\mu$ as a disjoint union of the partitions of $R$, as well as $\zeta^*_{i_0}$ and some leftovers columns. We claim that by the above procedure we can perform at most $(4r-4)C$ block moves on $\mu$ to modify all partitions $\zeta^*_i,\zeta^*_{i,j}$ in $R$ into $\zeta_i,\zeta_{i,j}$ without affecting $\zeta^*_{i_0}$. This is simple: we repeatedly move a block from the leftovers columns onto $\zeta^*_i$ or $\zeta^*_{i,j}$ in the manner described above. If the leftover column runs out, we instead create a brand new block to add to the partition being modified, so this procedure carries out the desired function.

Now we modify $\zeta^*_{i_0}$. We can add at most $C$ blocks from the remainder of the leftover columns or out of nowhere to $\zeta^*_{i_0}$. This forms $\zeta^{**}_{i_0}$ with $|\zeta^{**}_{i_0}|= |\varrho_{s_{i_0}}|$. By definition of $M$, we can then modify at most $M(s_{i_0})$ blocks in order to reach a partition $\zeta_{i_0}$ appearing in $\varrho_{s_{i_0}}^{\otimes 2}$. Because we assumed 
$$0\leq (\sum\limits_{i} |\varrho_{s_i}|)-|\mu|$$
if we used a leftover-column block whenever possible, we are now completely out of leftover column blocks, which means that our block moves have resulted in only the partitions $\zeta_i,\zeta_{i_0},\zeta_{i,j}$.

In all, we have made at most $(4r-4)+1$ sets of at most $C$ block moves each, as well as an additional $M(s_{i_0})\leq M(b)$. We now let $\widehat{\mu}$ the partition formed from the horizontal sum of the resulting $\zeta_i$, $\zeta_{i,j}$, and $\zeta_{i_0}$. Because we assumed $0\leq \sum\limits_{i} |\varrho_{s_i}|-|\mu|$, we must have consumed all of the leftover squares, so $\widehat{\mu}$ is precisely a horizontal sum of partitions comparable to the staircases $\varrho_{s_i}$. Using the semigroup property first to combine the modifications of $\zeta_{i,j}$ and then to combine all the $\zeta_i$, we have that $\widehat{\mu}$ is a constituent in the horizontal sum $$\hsum_{i=1}^r \varrho_{s_i}.$$  We made at most $(4r-3)C+M(b)$ block moves in transforming $\mu$ into $\widehat{\mu}$. Thus, part (1) is proved. 

The modification to prove part (2) is simple: assuming $X_{i_0} \leq \lceil \frac{b}{2}\rceil$ we go one step further and split $\varrho_{s_{i_0}}$ into 4 staircases, and split $\zeta_{i_0}$ as with the other $\zeta_i$. We now repeat the argument above just to these 4 parts $\zeta_{i_0,j}$: some parts have a small maximum height, and of those that do not, at most 1 fails to have a sufficiently small maximum height after another level of subdivision. The modification of the expression in the conclusion results from the subdivision of the exception case $i_0$ into smaller parts: there are an additional 12 pieces which need $C$ block moves each, but on the other hand the size of the exceptional piece is halved.

\end{proof}

We now get on with the main proof of Theorem~\ref{thm:linearH}. Fix an arbitrary partition $\mu\vdash n$.

Recall from Proposition~\ref{thm:layergrid} that we may write $\varrho_m$ as $$\varrho_m=(\varrho_x\Hplus ((k-1)\subV \varrho_y))\Vplus\left(\vsum_{i=0}^{k-1} \varrho_{z_i}\right)$$
where up to $O(1)$ error, we have $x\approx \frac{(k-1)m}{k},$ $  y\approx z_i\approx \frac{m}{k}$. By Proposition~\ref{thm:layersmooth} we may take $y, z_i$ pairwise differing by at most 1. The idea is to split $\mu$ into a large part corresponding to $\varrho_x$ and smaller parts corresponding to the other terms, then make some minor modifications to each part via block moves and apply the semigroup property. We will set $k=4$, so $x\approx\frac{3m}{4}$ and $  y\approx z_i\approx \frac{m}{4}$

We first split $\mu$ according to its Durfee square. Specifically, we may partition the blocks of $\mu$ into 3 pieces: the Durfee square $D$, the blocks $T_1$ to the right of $D$, and the blocks $T_2$ below $D$. WLOG, we may assume $|T_1| \geq |T_2|$, as otherwise we could conjugate $\mu$ and start over; since the staircases are symmetric, this doesn't affect anything. Let $D$ have side length $d$.

If the columns of $\mu$ are $c_1\geq c_2\geq \cdots\geq c_l$, consider the smallest $j$ such that $\sum_{i=1}^j c_i\geq |\varrho_x|$. Since $|T_1|\geq |T_2|$, we have that $\sum_{i=1}^{\left\lfloor\frac{d}{2}\right\rfloor} c_i \leq \frac{|D|}{2} + |T_2| \leq \frac{1}{2}n<|\varrho_x|\approx \frac{9n}{16}$. So $j> \frac{d}{2}$, and thus $n\geq jc_j > \frac{d}{2} c_j$, yielding $c_j < \frac{2n}{d}$.

Furthermore, if $d^2=|D|\leq \frac{1}{8}n$, then $j> d$, so $c_j\leq d\leq \frac{1}{2\sqrt{2}}\sqrt{n}$ by definition of the Durfee square. So either $c_j\leq \frac{1}{2\sqrt{2}}\sqrt{n}$, or $d>\frac{1}{2\sqrt{2}}\sqrt{n}$ and $c_j < \frac{2n}{d}<4\sqrt{2n}.$ So, $c_j<4\sqrt{2n}$ in either case.

We now split off the region of the partition corresponding to the first $j$ columns. 
Let the portion split off be $\lambda_0$, and let the remaining portion of the partition be $A$. By definition of $j$, we have $0\leq\left(3|\varrho_y|+\sum\limits_{i=0}^{3}|\varrho_{z_i}| \right)-|A| \leq c_j \leq 4\sqrt{2n}$.

To finish, we need to split $A$ into smaller staircase-sized partitions in such a way that makes the total number of block movements needed small. To this end we apply Lemma~\ref{lem:cuttail} with $r=2k-1$, where $s_1=\cdots=s_{k-1}=y$ and $s_{k+i}=z_i$ for $0\leq i\leq k-1$. Here we set $k=4$. Since the largest column involved has size $\mu_1'<4\sqrt{2n}$, we can set $C=4\sqrt{2n}$ and get

$$c\left(\hsum_{i=1}^r \varrho_{s_i},\hsum_{i=1}^r \varrho_{s_i},\widehat{\mu}\right)$$
for some $\widehat{\mu}$ such that $\Delta(A,\widehat{\mu})\leq 37C+M(\lceil\frac{y}{2}\rceil).$ Since $|A|\leq |\widehat{\mu}|$, $\mu=\lambda_0\Hplus A$ can be transformed into $\lambda_0^*\Hplus \widehat{\mu}$ by at most $37C+M(\lceil\frac{y}{2}\rceil)$ block moves. Because blockwise distance between equally sized partitions is simply the number of blocks which need to be moved to go from one to the other, this means that 

$$\tau_n^{\otimes (37C+M(\lceil\frac{y}{2}\rceil))}\otimes (\lambda^*_0\Hplus \widehat{\mu})$$
contains our original partition $\mu$, where $\lambda_0^*$ is some partition of size $|\varrho_x|$. 

By definition, there is a partition $\lambda_0^{**}$ of size $|\varrho_x|$ such that $c(\varrho_x,\varrho_x,\lambda_0^{**})$ and $\Delta(\lambda_0^{*},\lambda_0^{**})\leq M(x)$. So, by subadditivity of blockwise distance, as well as the layer decomposition for $\varrho_m$, we have

$$M(m)\leq M(x) + 37C+M\left(\left\lceil\frac{y}{2}\right\rceil\right)=M\left(\frac{3m}{4}+O(1)\right)+M\left(\frac{m}{8}+O(1)\right)+148m+O(1).$$

as $C=4\sqrt{2n}=4m+O(1)$. It is easy to see by strong induction that this recurrence gives $$M(m)\leq \frac{(148+o(1))}{(1-\frac{3}{4}-\frac{1}{8})}m=(1184+o(1))m.$$

\end{proof}

\subsection{Generalization to All $n$}

We now carry out the comparatively simple procedure of extending Theorem~\ref{thm:linearH} to non-triangular values of $n$. Since staircases can only have triangular sizes, we make a simple modification to the partitions $\varrho_m$ to give them the correct size.

\begin{defn}
The \emph{irregular staircase} $\xi_n$ of size $n=\frac{m(m+1)}{2}+k$, where $0\leq k\leq m$, is the partition $\xi_n=\varrho_m \Hplus 1_k$.
\end{defn}

So, an irregular staircase is a staircase with a trivial representation horizontally added to give it the desired total size. Note that $(m,k)$ are uniquely determined by $n$ in the above.

\begin{cor}\label{cor:close2all}
For all $n$, $\xi_n^{\otimes 2}\otimes \tau^{\otimes (1185+o(1))\sqrt{2n}}$ contains all partitions of $n$.
\end{cor}

\begin{proof}
We use Theorem~\ref{thm:linearH}. Given $n$, take the largest $m$ such that $\binom{m+1}{2}\leq n$. Let $k=n-\binom{m+1}{2}$, so $k\leq m$. Then $\xi_n = \varrho_{m} \Hplus 1_{k}$. We show that an arbitrary partition $\mu$ of $n$ may be transformed into a partition $\widehat{\mu}$ of $n-k$ horizontally summed with $1_k$, using $m$ block-movements. As $k\leq m$, it suffices to transform $\mu$ into a partition with at least $m$ parts of length 1. But this is trivial: repeatedly remove a block from a column of length at least 2, and move the block to create a column of length 1. If we can make $m$ such moves, we are done. If not, when we cannot make any more moves the current partition is a horizontal strip, and we are again done.

Now that this is shown, we have that $\mu$ is contained in $(\widehat{\mu} \Hplus 1_{k}) \otimes \tau_n^{\otimes m}$. By Theorem~\ref{thm:linearH} we have that $\widehat{\mu}$ is contained in $\varrho_m^{\otimes 2}\otimes \tau_{n-k} ^{\otimes (1184+o(1))m}$. Thus, there exists $\nu$ such that $c(\varrho_m,\varrho_m,\nu)$ with blockwise distance $\Delta(\nu,\widehat{\mu}) \leq (1184+o(1))m$. The semigroup property gives 

$$c(\xi_n,\xi_n,\nu+1_{k}).$$

We have $\Delta(\nu+1_{k},\widehat{\mu}+1_{k}) \leq (1184+o(1))m$ by Proposition~\ref{prop:distbound}. Since also $\Delta(\widehat{\mu}+1_{k},\mu)\leq m$ we conclude that $\Delta(\mu,\nu+1_{k})\leq (1185+o(1))m$. Since $\mu$ was arbitrary we have shown that

$$\xi_n^{\otimes 2}\otimes \tau^{\otimes (1185+o(1))m}$$
contains all partitions of $n$.

From $\sqrt{2n}=m+O(1)$ we conclude the desired result.
\end{proof}

\subsection{Replacing the Standard Representations}

\begin{lem}
\label{lem:stdInStair}
For any $\ell>0$, for large $n=\frac{m(m+1)}{2}$, if $\lambda$ is an irreducible representation of $S_n$ such that $\lambda$ is a component of $\tau_n^{\otimes \lfloor \ell m\rfloor}$, then $\lambda$ is a component of $\varrho_m^{\otimes 2}$.
\end{lem}

\begin{proof}
The irreducible components of $\tau_n^{\otimes\lfloor \ell m\rfloor}$ are precisely those that correspond to partitions $\lambda$ of $n$ whose blockwise distance from $1_n$ is at most $\lfloor \ell m\rfloor$.

We use Proposition~\ref{thm:layergrid} with $k=\ell_1 m^{\frac{1}{2}}$ for $\ell_1$ depending on $\ell$. This decomposes $\varrho_m$ into a large piece of side length about $m-\frac{m^{\frac{1}{2}}}{\ell_1}$, as well as about $2k-1$ small pieces of side length $\frac{m^{\frac{1}{2}}}{\ell_1}$ and area $\frac{m}{2(\ell_1)^2}$. 

Each column of $\lambda$ has size at most $\ell m+1$, and there are at least $n-2\ell m$ columns of height $1$. So, as long as $\frac{1}{\ell_1^2}>2\ell$, we can use each of the $2k-1$ small staircase pieces to cover, in their tensor square, a hook consisting of one of the $2k-1$ tallest columns along with part of the first row (chosen to give the correct total size). By an easy induction using the semigroup property (see \cite{IkenDom}), hooks are contained in the tensor square of staircases.

The remaining columns are necessarily very short: letting $c_1\geq c_2\geq \dots$ be the column lengths, we have $(2k)(c_{2k})\leq \sum\limits_{i=1}^{2k} c_i\leq \ell m+2k=O(m).$ Therefore, the remaining columns are all of size $O(m^{\frac{1}{2}})$. By Lemma~\ref{lem:domheights}, we conclude that this remaining large part of $\lambda$ is dominance comparable to the corresponding-size staircase. The semigroup property now yields the lemma.

\end{proof}

We can also extend Lemma~\ref{lem:stdInStair} to irregular staircases.
\begin{lem}
\label{lem:stdInIrrStair}
For any $\ell>0$, for large enough $n$, if $\lambda$ is an irreducible representation of $S_n$ such that $\lambda$ is a component of $\tau_n^{\otimes \lfloor \ell \sqrt{n}\rfloor}$, then $\lambda$ is a component of $\xi_n^{\otimes 2}$, the tensor square of the irregular staircase.
\end{lem}
\begin{proof}

The irreducible components of $\tau_n^{\otimes \lfloor \ell \sqrt{n}\rfloor}$ are precisely those that correspond to partitions $\lambda$ of $n$ whose blockwise distance from $1_n$ is at most $\lfloor \ell \sqrt{n}\rfloor$. So, there are at least $n-2\ell \sqrt{n}$ columns of height $1$. Again take the largest $m$ such that $\binom{m+1}{2}\leq n$, and let $k=n-\binom{m+1}{2}$. For large enough $n$, we can remove $k$ columns of size $1$ and leave a partition $\lambda_0$ of $n_0=\binom{m+1}{2}$ that has distance at most $\ell\sqrt{n}\leq (1.1)\ell \sqrt{n_0}$ from the trivial representation of size $n_0$. By Lemma~\ref{lem:stdInStair}, for large enough $n$, all such $\lambda$ are contained in $\varrho_m^{\otimes 2}$, and so by the semigroup property, after adding the trivial representation back, we must have $\xi_n^{\otimes 2}$ contains $\lambda$.
\end{proof}

We can finally prove Theorem~\ref{thm:4thpower} on tensor fourth powers, with an explicit choice $\lambda=\xi_n$.

\begin{repthm}{thm:4thpower}
For sufficiently large $n$, the tensor fourth power $\xi_n^{\otimes 4}$ contains all partitions of $n$.
\end{repthm}

\begin{proof}

Corollary~\ref{cor:close2all} implies that for sufficiently large $n$, $\xi_n^{\otimes 2}\otimes \tau_n^{\otimes \lfloor 1186\sqrt{2n}\rfloor}$ contains all partitions of $n$. So for any $\mu\vdash n$, $\xi_n^{\otimes 2}\otimes \nu$ contains $\mu$ for some $\nu$ contained in $\tau_n^{\otimes \lfloor 1186\sqrt{2n}\rfloor}$. Lemma~\ref{lem:stdInIrrStair} thus implies that $\mu$ is contained in $\xi_n^{\otimes 4}$, as claimed.

\end{proof}

\section{Concluding Remarks}

We have shown that staircase tensor squares $\varrho_m^{\otimes 2}$ contain almost all partitions in 2 natural probability distributions, and that there are partitions $\lambda\vdash n$ for all large $n$ such that $\lambda^{\otimes 4}$ contains all partitions of $n$.

\begin{rem}
The argument used to show that tensor 4th powers contain every representation proceeds by showing first that every partition $\lambda$ is near another partition contained in the tensor square $\varrho^{\otimes 2}$, and then shows that the nearness can be encompassed in another $\varrho^{\otimes 2}$ factor. To prove the full tensor square conjecture using our semigroup methods, one would need to remove all of the standard representations, which means each semigroup property application would need to be exactly correct. As a result, improving the exponent from 4 to 2 seems more difficult. Intuitively, our value $4$ is really $ 2\lceil 1+\varepsilon \rceil $.
\end{rem}

\begin{rem}

Our results in this paper focused on the staircase partitions, but many of the arguments can be adapted for partitions which can be similarly broken down into staircase pieces. For instance, the caret partition $\gamma_k$ mentioned in the introduction may be expressed as $\gamma_k=(\varrho_{2k}\Hplus\varrho_{k-1})\Vplus\varrho_{k-1}$. By breaking down the $\varrho_{2k}$ into 4 pieces, we have broken $\gamma_k$ into 6 approximately equal staircases. As a result, the proof of Theorem~\ref{thm:planch} on Plancherel-random partitions applies to $\gamma_k^{\otimes 2}$ as well. 

\end{rem}

\begin{rem}
\label{rem:recthard}
Intuitively, the rectangular Young diagrams should be the most difficult to deal with using the semigroup property: if $\lambda$ is rectangular and $\lambda=\lambda_1\Hplus\lambda_2$, both $\lambda_1$ and $\lambda_2$ must be rectangular. Therefore, we have very little freedom in applying the semigroup property. It is, however, easy to show that rectangles are constituents of the tensor \textit{cubes} $\varrho_m^{\otimes 3}$; see Appendix~\ref{subsec:rectcube} for details. Further exploration could yield more insight on how difficult general partitions are to fit into a tensor cube.
\end{rem}

\begin{rem}

The question of whether the semigroup property combined with combinatorial arguments involving dominance and symmetry suffices to prove the full tensor square conjecture remains open. Regardless, checking for Kronecker coefficient positivity inductively using the semigroup property seems much faster than directly computing Kronecker coefficients. Of course, the semigroup property may fail to detect positive Kronecker coefficients.

\end{rem}

\begin{rem}
\label{rem:bash}
We have, using a computer to implement the semigroup property in conjunction with Theorem~\ref{thm:dominance} on dominance ordering, verified the Saxl conjecture up to $\varrho_9$. These two facts suffice for all cases except the 6 by 6 square in $\varrho_8^{\otimes 2}$. This case follows from the semigroup property using the additional fact that $c(\lambda,\lambda,\lambda)$ holds for every symmetric partition $\lambda$ (\cite{symcube}), but our construction seems rather ad-hoc. We explain it here using some helpful visuals. In the diagram (below), each color corresponds to an application of the semigroup property, beginning with the red squares (which satisfy constituency by the theorem mentioned above). Thus, the 3 depicted partitions yield a positive Kronecker coefficient. 

\begin{figure}[h]
\label{fig:6x6step1}
\caption{A Triple with Positive Kronecker Coefficient}
\includegraphics[scale=0.8]{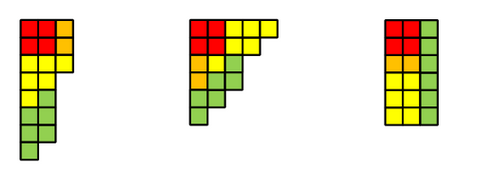} 
\end{figure}

We now add the rectangle to itself, and add the other 2 partitions to each other, giving $c(\varrho_8,\varrho_8,\mu)$ for $\mu$ the 6 by 6 square.

\comment{\begin{figure}[h]
\label{fig:6x6step2}
\caption{Permuting and adding the previous triple to itself proves that this triple also has a positive Kronecker coefficient.}
\includegraphics[scale=0.8]{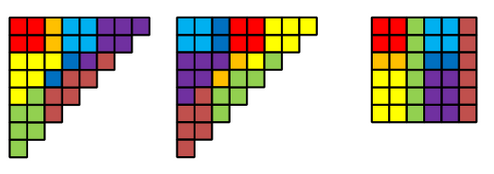} 
\end{figure}}

\end{rem}

\section{Acknowledgements}

The authors thank their MIT SPUR 2015 mentor, Dongkwan Kim, for his valuable feedback and advice. They also thank Nathan Harman for suggesting this problem, and professors Alexei Borodin, David Jerison, and Ankur Moitra for their helpful input and advice. They thank Mitchell Lee for help with editing, and finally thank the MIT SPUR program for the opportunity to conduct this research project.

\appendix \section{Appendix: Technical Lemma on $\beta$-Sum Flexibility}
\normalsize
\subsection{Overview of the Proof}
\label{subsec:bessoverview}

We conclude by proving the key technical result used in the proof of Theorem~\ref{thm:planch}. This result is restated below:
\begin{repthm}{thm:bess}
For any $\beta>0$, let $P(n,\beta)$ denote the probability that a (Plancherel) random partition of $n$ is $\beta$-sum flexible. Then we have
\[\lim_{n\rightarrow\infty} P(n,\beta) = 1\] for all $\beta$.
\end{repthm}

We will estimate the typical value for the maximum height among the smallest $n^{\alpha}$ columns, for $\frac{1}{6} < \alpha \leq \frac{1}{2}$. These bounds will enable us to conclude Theorem~\ref{thm:bess}. We first recall from, e.g., \cite{Borodin} that, for a random partition of $n$, the number $\lambda_1-\lambda_2$ of columns of size $1$ is $\Theta(n^{1/6})$ with high probability. In fact, $\frac{\lambda_1-\lambda_2}{n^{1/6}}$ converges weakly to a non-trivial limiting density.

This result tells us that for any $\epsilon$ there exists $\delta$ such that, for large $n$, we will have at least $\delta n^{1/6}$ parts of size 1 with probability at least $1-\epsilon$. This means that the condition in Theorem~\ref{thm:bess} is safe for all columns of size at most $\delta n^{1/6}$. If the sum of all these columns is of size $\Theta({n^{\alpha}})$, the condition in Theorem~\ref{thm:bess} is now safe for all columns with at most $c n^{\alpha}$ size, for some c. Our plan is to ``bootstrap" in this manner up to $\Theta({n^{1/2}})$. If we can achieve the condition up to this point, we will have proved (4), because for pieces of size $\Theta({n^{1/2}})$ the limit shape easily implies Theorem~\ref{thm:bess}. 

We actually estimate a closely related quantity, for which explicit formulae exist. For a partition $\lambda$, we follow \cite{Borodin} in denoting by $D(\lambda)$ the set $\{\lambda_i - i\}$. We will also work with poissonized Plancherel measure $M^{\theta}$ instead of ordinary Plancherel measure $M_n$. At the end, we will depoissonize to recover information about the measures $M_n$. The use of Poissonized Plancherel measure is that useful exact formulae describe the behavior of a random partition.

\begin{defn}
For $\theta \in \mathbb{R}^+$, the Poissonized Plancherel measure $M^{\theta}$ is a probability distribution over \textit{all} partitions $\lambda$, with $$M^{\theta}(\lambda)=e^{-\theta}\theta^{|\lambda|}\left(\frac{\dim(\lambda)}{|\lambda|!}\right)^2.$$
\end{defn}

Conceptually, we pick a $M^{\theta}$ partition by first taking $n$ to be Poisson with mean $\theta$, and then picking a Plancherel-random $\lambda\vdash n$.

The formulas to follow will allow us to estimate the mean and variance $\mu_{\theta, w}, \sigma_{\theta,w}^2$ of $T(\lambda,w)=|D(\lambda) \cap \mathbb{Z}_{\geq w}|$, where $\lambda$ is taken from poissonized Plancherel measure with mean $\theta$. Because the set $D(\lambda)$ corresponds to the vertical border edges of the Young diagram, this approximately gives the height of the $w$th smallest column. By setting $w=2\sqrt{n}-Kn^{\alpha}$, knowledge of these quantities will allow us to understand roughly the value of the maximum-size column among the first $\Theta(n^{\alpha})$. 

We now list the formulas and bounds we will use for these computations.

\subsection{Formulas for Poissonized Plancherel Measure}
\label{subsec:bess}

\begin{defn}
\label{defn:bess}
The \textit{Bessel function of the first kind} $J_{\nu}$ is defined as
$$J_{\nu}(x)=\sum\limits_{m=0}^{\infty} \frac{(-1)^m (\frac{x}{2})^{2m+\nu}}{(m!)\Gamma({m+\nu+1})}.$$
\end{defn}

\begin{defn}
\label{defn:airy}
The \textit{Airy function} $Ai$ is defined as
$$Ai(x)=\frac{1}{\pi} \int_{0}^{\infty} \cos\left(\frac{u^3}{3}+xu\right)du.$$
\end{defn}

\begin{defn}
\label{defn:descents}
For a partition $\lambda=(\lambda_1,\lambda_2,\dots$, define $D(\lambda)$ to be the set $\{\lambda_i-i\}$.
\end{defn}

\begin{defn}
Define the function $\textbf J(x,y;\theta)$ as 

$$\textbf{J}(x,y;\theta)=\sum\limits_{s=1}^{\infty}  J_{x+s}(2\sqrt{\theta})J_{y+s}(2\sqrt{\theta}).$$
\end{defn}

\begin{lem} [{\cite[Thm~2, Prop~2.9]{Borodin}}]
\label{bess:formula}
Let $\lambda$ be chosen according to $M^{\theta}$ . Then for any finite $X=\{x_1,x_2\dots x_s\} \subseteq \mathbb{Z}$, the probability that $D(\lambda) \supseteq X$ is

$$ M^{\theta}\left(\{\lambda | D(\lambda) \supseteq X\}\right)=\det\left[\textbf{J}(x_i,x_j,\theta)\right]_{1\leq i,j \leq s}.$$
\end{lem}

 \begin{lem} [{\cite{bess}}]
 \label{bess:bound}
 $J_{\nu}(x) \leq x^{-1/3}$ and 
 $J_{\nu}(x) \leq \nu^{-1/3}$.
 \end{lem}
 
 \begin{lem} [{\cite[Lemma~4.4]{Borodin}}]
 \label{bess:airy} For $x \in \mathbb{R}$ we have $$\left|n^{1/6} J_{2n^{1/2}+xn^{1/6}}(2n^{1/2})-Ai(x)\right| = O\left(n^{-1/6}\right), \hspace{5 mm} n \rightarrow \infty$$ where additionally the implicit constant in $O(n^{-1/6})$ is uniform for $x$ in any compact set.
 \end{lem}
 
 \begin{lem} [{\cite[Prop~4.3]{Borodin}}]
For fixed $a \in \mathbb{R}$,
 \[\lim_{n\rightarrow \infty}
\left(\sum\limits_{k=2\sqrt{n}+an^{\frac{1}{6}}}^{\infty} J(k,k;n)\right) = \int_0^{\infty} t\left(Ai(a+t)\right)^2 dt. \]
\end{lem}
 
Specializing to $a=0$ gives
 
 \begin{lem}
 \label{bess:bounded} \[\lim_{n\rightarrow \infty} \left(\sum\limits_{k=2\sqrt{n}}^{\infty} J(k,k;n)\right) = \int_0^{\infty} t(Ai(t))^2 dt. \]
\end{lem}

\begin{lem}[{\cite[Lemma~4.5]{Borodin}}]
\label{bess:originaltailbound} 
There exist $C_1, C_2, C_3, \varepsilon > 0$ such that for sufficiently large $n$, for any $A>0$, $s > 0$, we have  

$$\left|J_{r+Ar^{1/3}+s}(r)\right| \leq C_1 r^{-1/3} \exp\left(-C_2(A^{\frac{3}{2}}+sA^{\frac{1}{2}}r^{-1/3})\right), \hspace{2mm} s \leq \varepsilon$$ 
$$\left|J_{r+Ar^{1/3}+s}(r)\right| \leq \exp\left(-C_3(r+s)\right),\hspace{2mm} s \geq \varepsilon$$

for all $r\gg0$.
\end{lem}

Setting $r=2n^{\frac{1}{2}}$ and $A=2^{-\frac{1}{3}}$ in Lemma~\ref{bess:originaltailbound}, squaring, and adjusting the values $C_i$ as needed, we have the following for all $s$, for some fixed values $C_i, \varepsilon$, and large enough $n$.

\begin{lem}
\label{bess:tailbound}

There exist $C_1, C_2, C_3, \varepsilon > 0$ such that for sufficiently large $n$, for any $s > 0$, we have  

$$\left(J_{2n^{1/2}+n^{1/6}+s}(2n^{1/2})\right)^2 \leq C_1 n^{-1/3} \exp\left(-C_2(1+sn^{-1/6})\right), \hspace{2mm} s \leq \varepsilon n^{1/2}$$

$$ \left(J_{2n^{1/2}+n^{1/6}+s}(2n^{1/2})\right)^2 \leq \exp\left(-C_3(2n^{1/2}+s)\right), \hspace{2mm} s \geq \varepsilon n^{1/2}.$$\end{lem}

\subsection{Proof of Theorem~\ref{thm:bess}}

We proceed as described in Section~\ref{subsec:bessoverview}. First, direct application of Lemma~\ref{bess:formula} in the case $|X|=1$, combined with linearity of expectation, yields that $$\mu_{\theta, w} = \sum\limits_{y=w}^{\infty}  J(y,y,\theta)=\sum\limits_{y=w}^{\infty} \sum\limits_{s=1}^{\infty} (J_{y+s}(2\sqrt{\theta}))^2 = \sum\limits_{r=1}^{\infty} r (J_{w+r}(2\sqrt{\theta}))^2, $$  
where the last equality is a simple sum rearrangement. Now let $\theta=n, w=2n^{1/2}-Kn^{\alpha}, \frac{1}{6} < \alpha \leq \frac{1}{2}$, as alluded to before. We will now establish the bounds on this expectation.

\begin{lem}
\label{lem:lowerbound}
$\mu_{n,(2n^{1/2}-Kn^{\alpha})}=\Omega\left(n^{\alpha-\frac{1}{6}}\right)$.
\end{lem}

\begin{proof}Taking $n$ sufficiently large, we have that $$\mu_{n,(2n^{1/2}-Kn^{\alpha})} = \sum\limits_{r=1}^{\infty} r (J_{2n^{1/2}-Kn^{\alpha}+r}(2\sqrt{n}))^2 \geq \sum\limits_{r=Kn^{\alpha}-n^{1/6}}^{Kn^{\alpha}+n^{1/6}} r (J_{2n^{1/2}-Kn^{\alpha}+r}(2\sqrt{n}))^2 $$
$$ \geq \sum\limits_{j=-n^{1/6}}^{n^{1/6}} \left(\frac{Kn^{\alpha}}{2}\right) (J_{2n^{1/2}+j}(2\sqrt{n}))^2 = \left(\frac{Kn^{\alpha}}{2}\right) \sum\limits_{j=-n^{1/6}}^{n^{1/6}} (J_{2n^{1/2}+j}(2\sqrt{n}))^2. $$

Because the Airy function has no zeros in the interval $[-1,1]$ and is continuous, the uniformity in Lemma~\ref{bess:airy}  implies that each squared term in this last sum is of size $\Theta(n^{-1/3})$. Since we have $\Theta(n^{1/6})$ terms in the sum, we conclude that for large n, $$\mu_{n,(2n^{1/2}-Kn^{\alpha})}\geq \left(\frac{Kn^{\alpha}}{2}\right) \sum\limits_{j=-n^{1/6}}^{n^{1/6}} (J_{2n^{1/2}+j}(2\sqrt{n}))^2 =\Theta(n^{\alpha-\frac{1}{6}}). \hspace{1cm} $$ 
\end{proof}

\begin{lem}
\label{lem:upperbound}
$\mu_{n,(2n^{1/2}-Kn^{\alpha})}=O\left(n^{2\alpha-\frac{1}{3}}\right)$. Further, given fixed $\alpha$, the implicit constant in $O(n^{2\alpha-\frac{1}{3}})$ is $O(K^2)$
\end{lem}

\begin{proof}
Using Lemma~\ref{bess:bounded}, we see that $$\mu_{n,(2n^{1/2}-Kn^{\alpha})}= \sum\limits_{y=2n^{1/2}-Kn^{\alpha}}^{\infty}  J(y,y,n) = \sum\limits_{y=2n^{1/2}-Kn^{\alpha}}^{2n^{1/2}}  J(y,y,n)+\sum\limits_{y=2n^{1/2}}^{\infty}  J(y,y,n) $$$$= \sum\limits_{y=2n^{1/2}-Kn^{\alpha}}^{2n^{1/2}}  J(y,y,n)+O(1). $$

Expanding out the definition of $\textbf{J}(y,y,n)$ we have that 
$$\sum\limits_{y=(2n^{1/2}-Kn^{\alpha}}^{2n^{1/2})}  \textbf J(y,y,n) = \sum\limits_{y=(2n^{1/2}-Kn^{\alpha})}^{2n^{1/2}} \sum\limits_{s=1}^{\infty} (J_{y+s+1}(2n^{1/2}))^2 =$$
$$ \left( \sum\limits_{\ell=1}^{Kn^{\alpha}} \ell(J_{\ell+s+1}(2n^{1/2}))^2 \right) + Kn^{\alpha} \sum\limits_{m=2n^{1/2}}^{\infty} (J_m(2n^{1/2}))^2,$$
where the last equality follows from another simple regrouping of terms. By Lemma~\ref{bess:bound}, for any value of $\ell$ we have $J_{\ell+s+1}(2n^{1/2}) \leq n^{-1/6}$. Thus, we have that the first sum is of appropriate size: 
$$ \left( \sum\limits_{\ell=1}^{Kn^{\alpha}} \ell(J_{\ell+s+1}(2n^{1/2}))^2 \right) \leq  \left( \sum\limits_{\ell=1}^{Kn^{\alpha}} \ell \right) n^{-1/3} \leq K^2n^{2\alpha-\frac{1}{3}}.$$

To establish Lemma~\ref{lem:upperbound}, it remains to show that the latter term $$Kn^{\alpha} \sum\limits_{m=2n^{1/2}}^{\infty} (J_m(2n^{1/2}))^2$$ is of appropriate size. First, note that from Lemma~\ref{bess:bound} again,  $$Kn^{\alpha}\sum\limits_{m=2n^{1/2}}^{2n^{1/2}+n^{1/6}} (J_m(2n^{1/2}))^2 \leq Kn^{\alpha}n^{1/6} (n^{-1/3}) = Kn^{\alpha-\frac{1}{6}} = O(n^{2\alpha-\frac{1}{3}}),$$ where the last equality follows from the assumption $\alpha > \frac{1}{6}$.
We are left with upper-bounding the sum $$Kn^{\alpha}\sum\limits_{m=2n^{1/2}+n^{1/6}}^{\infty} (J_m(2n^{1/2}))^2.$$
Now, using Lemma~\ref{bess:tailbound}, take suitable constants $C_1, C_2, C_3, \epsilon$, and let $n$ be large enough, so that we have for all s
$$ (J_{2n^{1/2}+n^{1/6}+s}(2n^{1/2}))^2 \leq C_1 n^{-1/3} \exp(-C_2(1+sn^{-1/6})), \hspace{2mm} s \leq \epsilon n^{1/2}$$
$$ (J_{2n^{1/2}+n^{1/6}+s}(2n^{1/2}))^2 \leq \exp(-C_3(2n^{1/2}+s)), \hspace{2mm} s \geq \epsilon n^{1/2}.$$

We claim that the above sum is also $O(n^{\alpha-\frac{1}{6}})$, or equivalently that
$$\sum\limits_{m=2n^{1/2}+n^{1/6}}^{\infty} (J_m(2n^{1/2}))^2= \sum\limits_{m=2n^{1/2}+n^{1/6}}^{(2+\epsilon)n^{1/2}+n^{1/6}} (J_m(2n^{1/2}))^2+\sum\limits_{m=(2+\epsilon)n^{1/2}+n^{1/6}}^{\infty} (J_m(2n^{1/2}))^2 = O(n^{-1/6}).$$

For the first of these 2 sums, we have 
$$\sum\limits_{m=2n^{1/2}+n^{1/6}}^{(2+\epsilon)n^{1/2}+n^{1/6}} (J_m(2n^{1/2}))^2=
\sum\limits_{s=0}^{\epsilon n^{1/2}} (J_{2n^{1/2}+n^{1/6}+s}(2n^{1/2}))^2 \leq
\sum\limits_{s=0}^{\epsilon n^{1/2}} C_1 n^{-1/3} \exp(-C_2(1+sn^{-1/6})) $$ 
$$ \leq C_1 n^{-1/3}\exp(-C_2) \sum\limits_{s=0}^{\infty} (\exp(-C_2(n^{-1/6})))^s = \frac{C_1 n^{-1/3}\exp(-C_2)}{1-\exp(-C_2n^{-1/6})}$$
$$=\Theta(\frac{C_1n^{-1/3}\exp(-C_2)}{C_2n^{-1/6}}) = \Theta(\frac{C_1\exp(-C_2)}{C_2} n^{-1/6})=O(n^{-1/6}).$$

For the second, we have 
$$\sum\limits_{m=(2+\epsilon)n^{1/2}+n^{1/6}}^{\infty} (J_m(2n^{1/2}))^2=\sum\limits_{s=\epsilon n^{1/2}}^{\infty} (J_{2n^{1/2}+n^{1/6}+s}(2n^{1/2}))^2 \leq \sum\limits_{s=\epsilon n^{1/2}}^{\infty}  \exp(-C_3(2n^{1/2}+s)) \leq $$
$$ \sum\limits_{s=0 }^{\infty}\exp(-C_3(2n^{1/2}+s)) = \frac{\exp(-2C_3n^{1/2})}{1-exp(-C_3)} = \exp(\Theta(-n^{1/2}))=O(n^{-1/6}).$$

Note that only the first part of our bounding affects the eventual constant in the $O(n^{2\alpha-\frac{1}{3}})$ of the lemma statement, because the other terms are $O(n^{\alpha-\frac{1}{6}})=o(n^{2\alpha-\frac{1}{3}})$. So, combining these separate bounds, we have established the lemma. 
\end{proof}

Now we verify that $T(n,2n^{1/2}-Kn^{\alpha})$ is concentrated around its mean.

\begin{lem}
\label{lem:concentrated}
If $\lambda$ is distributed according to poissonized $M^n$, then $T=T(\lambda,2n^{1/2}-Kn^{\alpha})$ is concentrated near its mean $\mu=\mu_{n,(2n^{1/2}-Kn^{\alpha})}$, in the sense that for all $\epsilon > 0$,  $$\lim_{n \rightarrow \infty} \mathbb{P}[|T-\mu|>\epsilon\mu|] =0.$$
\end{lem}

\begin{proof}
To show the lemma, we estimate the variance $\sigma^2$ of $T(\lambda,2n^{1/2}-Kn^{\alpha})$. Perhaps surprisingly, this is very easy. We will show that $$\sigma^2=\sigma_{n,(2n^{1/2}-Kn^{\alpha})}^2 \leq \mu= \mu_{n,(2n^{1/2}-Kn^{\alpha})},$$
which suffices by the Chebyshev inequality, since $\mu$ grows to infinity with n by Lemma~\ref{lem:lowerbound}. To show this, we note the following general fact (Lemma~\ref{lem:anticorr}): if $X=\sum\limits_{j=0}^{\infty}x_j$ is a finite-expectation sum of Bernoulli (0 or 1) random variables with pairwise non-positive covariances, then $\sigma_X^2 \leq \mu_X$.

In our case, this lemma easily implies the result: when $\lambda$ is distributed according to poissonized $M^n$ we have $$T(\lambda,2n^{1/2}-Kn^{\alpha})=\sum\limits_{w=2n^{1/2}-Kn^{\alpha}}^{\infty} I(w \in D(\lambda)),$$ where the indicator variables $I$ are 0 or 1 according to the truth value of their argument. We need only check that $Cov(I(x \in D(\lambda)),I(y \in D(\lambda))) \leq 0$ for $x \neq y$. We have $$Cov(I(x \in D(\lambda)),I(y \in D(\lambda)))= \mathbb{P}[\{x,y\} \subseteq D(\lambda)]-\mathbb{P}[x \in D(\lambda)]\mathbb{P}[y \in D(\lambda)].$$
Using Lemma~\ref{bess:formula}, this is 
$$\det{\left(\begin{matrix} J(x,x;n) &J(x,y;n) \\ J(y,x;n)&J(y,y;n)\end{matrix}\right)} - J(x,x;n)J(y,y;n)= -J(x,y;n)J(y,x;n)=-J(x,y;n)^2\leq 0,$$
since $J(x,y;n)$ is symmetric in $x$ and $y$. Therefore, the covariances are all negative, so it only remains to verify Lemma~\ref{lem:anticorr}
\end{proof}
\begin{lem}
\label{lem:anticorr}
$X=\sum\limits_{j=0}^{\infty}x_j$ is a finite-expectation sum of Bernoulli variables with pairwise non-positive covariances, then $\sigma_X^2 \leq \mu_X$
\end{lem}
\begin{proof}
To prove this, we simply compute $\sigma_X^2=\mathbb{E}[X^2]-\mathbb{E}[X]^2$. The condition on covariances is equivalent to $\E[x_ix_j] \leq \E[x_i]\E[x_j]$ for all $i \neq j$. Because $\E[x]=\sum\limits_{i=0}^{\infty} \E[x_i]$ is finite, its square is finite, so we have
$$(\E[x])^2=\left(\sum\limits_{i=0}^{\infty} (\E[x_i])^2\right) + \left(\sum\limits_{i>j \geq 0}^{} \E[x_i]\E[x_j]\right) < \infty.$$
Because $\E[x_ix_j] \leq \E[x_i]\E[x_j]$, we find that
$$\E[x]^2+\mu_X=\E[x]^2+\E[x]=\left(\sum\limits_{i=0}^{\infty} (\E[x_i])^2+\E[x_i]\right) + \left(\sum\limits_{i>j \geq 0}^{} \E[x_i]\E[x_j]\right)$$
$$
\geq \left(\sum\limits_{i=0}^{\infty} (\E[x_i^2])\right) + \left(\sum\limits_{i>j \geq 0}^{} \E[x_ix_j]\right)=\E[x^2].$$

\end{proof}

%\subsection{Depoissonization}

We now know that the bounds in lemmas~\ref{lem:lowerbound},~\ref{lem:upperbound} apply to almost all partitions $\lambda$, not only ``average" partitions: for some constants $B_1(K,\alpha), B_2(K,\alpha)$ where $B_2=O(K^2)$ for fixed $\alpha$, if $\lambda$ is distributed according to $M^n$ we have $$B_1(K,\alpha)n^{\alpha-\frac{1}{6}} \leq T(\lambda,2n^{1/2}-Kn^{\alpha}) \leq B_2(K,\alpha)n^{2\alpha-\frac{1}{3}}$$
with probability $1-o(1)$ as $n \rightarrow \infty$. We now depoissonize the result. The key observation, as described in \cite{Fulman}, is that simple Plancherel measures $M_n$ may be viewed as samplings at time $n$ of a growth process on partitions, in which we begin with an empty partition and add additional blocks randomly at each step. It is clear that adding a block to $\lambda$ cannot decrease the value of $T(\lambda,w)$. So because we can embed these Plancherel measures into a growth process, the expectations are monotone: for $\ell<n$ we have

$$\mathbb{P}\left[\lambda^{(\ell)} < B_1(K,\alpha)n^{\alpha-\frac{1}{6}}\right] \geq \mathbb{P}\left[\lambda^{(n)} < B_1(K,\alpha)n^{\alpha-\frac{1}{6}}\right] $$
and for $\ell > n$ we have 

$$\mathbb{P}\left[\lambda^{(\ell)} > B_2(K,\alpha)n^{2\alpha-\frac{1}{3}}\right] \geq \mathbb{P}\left[\lambda^{(n)} > B_2(K,\alpha)n^{2\alpha-\frac{1}{3}}\right]. $$

We claim that, for $n$ large, the probability of a Poisson variable with mean $n$ to be less than $n$ converges to $\frac{1}{2}$. Indeed, this is an immediate result of the Central Limit Theorem for independent, identically distributed random variables applied to the Poisson random variable with mean 1. Therefore, \[ \limsup_{n \rightarrow \infty} M_n\left(\left\{\lambda|T(\lambda,2n^{1/2}-Kn^{\alpha}) < B_1(K,\alpha)n^{\alpha-\frac{1}{6}}\right\}\right) \]  
\[ \leq 2 \limsup_{n \rightarrow \infty} M^n\left(\left\{\lambda|T(\lambda,2n^{1/2}-Kn^{\alpha}) < B_1(K,\alpha)n^{\alpha-\frac{1}{6}}\right\}\right) =0,\]
and similarly for the upper bound; a low probability of deviation for poissonized Plancherel measure directly implies the same bound (up to a factor of 2) for standard Plancherel measure. So, we have 
$$B_1(K,\alpha)n^{\alpha-\frac{1}{6}} \leq T(\lambda,2n^{1/2}-Kn^{\alpha}) \leq B_2(K,\alpha)n^{2\alpha-\frac{1}{3}}$$
with probability $1-o(1)$ for $\lambda$ distributed according to $M_n$ measure. 

%\subsection{Finishing the argument}

Now we establish the connection between $T(\lambda,2n^{1/2}-Kn^{\alpha})$ and the column sizes. It will be valuable to visualize a Young diagram (in English coordinates) as a block-walk beginning from the far-right of the top border line, with coordinates $(n,0)$. In this model, every down-move corresponds to a value $(\lambda_i-i)$ and every left-move corresponds to the end of a column. With this in mind, $T(\lambda,2n^{1/2}-Kn^{\alpha})$ is essentially the height of the largest column so far after $n-2n^{1/2}+Kn^{\alpha}$ moves from the starting point. The idea is that our upper bound lets us move far out without fearing that our largest column is too big. Our lower bound tells us that many of the columns we have formed are large, implying that the sum of their sizes is large. The only caveat is that if we are moving vertically, $T$ measures the current vertical distance moved, but may not exactly measure a column height; we may be in the middle of a column. However, we will circumvent this issue by using multiple values of $K$ simultaneously, and showing that in between there must be many horizontal moves. 

Pick an $\epsilon > 0$. We will give a bootstrapping argument that gives result (4) for parameter $\beta$ with probability $1-6\epsilon$. First, because the number of size-$1$ columns $\lambda_1-\lambda_2$ satisfies $\frac{\lambda_1-\lambda_2}{n^{1/6}}$ converges in distribution to a limiting density, there exists $\delta_1$ such that (for large $n$) with probability at least $1-\epsilon$, there are at least $\frac{\delta_1}{\beta}n^{1/6}$ parts of size 1. Using lemmas~\ref{lem:lowerbound} and~\ref{lem:upperbound}, since $\frac{1}{4}-\frac{1}{6}=\frac{1}{12}$ and $\frac{2}{4}-\frac{1}{3}=\frac{1}{6}$, we may pick $\delta_2,$ and then $ \delta_3$, such that for large $n$, with probability at least $1-\epsilon$, $$T(\lambda,2n^{1/2}-\delta_2n^{1/4}) \geq {\delta_3}n^{\frac{1}{12}}$$
and $$T(\lambda,2n^{1/2}-2\delta_2n^{1/4}) \leq \delta_1n^{\frac{1}{6}}$$
(for the second inequality, we use the fact that the constant factor in Lemma~\ref{lem:upperbound} is $O(K^2)$). Now, for large $n$, we have $\delta_2n^{1/4} -\delta_1n^{\frac{1}{6}} \geq \frac{\delta_2}{2}n^{1/4}$. This implies that $D(\lambda)$ cannot contain more than $\frac{\delta_2}{2}n^{1/4}$ of the interval of integers $[2n^{1/2}-2\delta_2n^{1/4}, 2n^{1/2}-\delta_2n^{1/4}]$ as this would violate the latter inequality above. In view of our block-walking model, this is equivalent to at least $\frac{\delta_2}{2}n^{1/4}$ horizontal move being made in between the steps $n-2n^{1/2}+\delta_2n^{1/4}, n-2n^{1/2}+2\delta_2n^{1/4}$ meaning that at least $\frac{\delta_2}{2}n^{1/4}$ columns in this range have height at least $\delta_3n^{\frac{1}{12}}$ (this lower bound on the height comes from the first inequality above). Because the largest column among these is smaller than $\delta_1n^{\frac{1}{6}}$, the condition in result (4) holds for all columns up to size $\delta_1n^{\frac{1}{6}}$. Therefore, by bootstrapping, it holds for all columns with size at most $\beta$ times the sum of their sizes, or at most 
$$\beta\left(\frac{\delta_2}{2}n^{\frac{1}{4}}\right)\left(\delta_3n^{\frac{1}{12}}\right)=\frac{\beta\delta_2\delta_3}{2}n^{\frac{1}{3}}=\delta_4n^\frac{1}{3}.$$

Now we repeat the argument once more. Again using Lemma~\ref{lem:upperbound}, we pick $\delta_5$ such that with probability $1-\epsilon$,
$$T\left(\lambda, 2n^{\frac{1}{2}}-2\delta_5n^{\frac{1}{3}}\right) \leq \frac{\delta_5}{2}n^{\frac{1}{3}}$$
and $$\delta_5 \leq \delta_4.$$
The first is possible because the constant in Lemma~\ref{lem:upperbound} is $O(K^2)$, implying that it is $o(K)$ for $K$ small. Given $\delta_5$ we can then pick $\delta_6$ such that with probability $1-\epsilon$, $$T(\lambda,2n^{1/2}-\delta_5n^{\frac{1}{3}})\geq \delta_6n^{1/6}.$$

The first inequality above implies that of the $\delta_5n^{1/3}$ block moves from step $n-2n^{1/2}+\delta_5n^{1/3}$ to step $n-2n^{1/2}+2\delta_5n^{1/3}$, we may make at most $\frac{\delta_5}{2}n^{1/3}$ vertical moves. Hence, we make at least $\frac{\delta_5}{2}n^{1/3}$ horizontal moves in this span. By the definition of $\delta_6$, we make a column of size at least $\delta_6n^{\frac{1}{6}}$ each time, giving a total column size sum of $\frac{\delta_5\delta_6}{2}n^{1/2}$. 

Now we are almost done. Our bootstrapping has shown Theorem~\ref{thm:bess} for columns of size at most $\frac{\beta\delta_5\delta_6}{2}=\delta_7n^{1/2}$
with probability at least $1-4\epsilon$. Bootstrapping once more, we can now sum up all the column lengths of size at most $\delta_7n^{1/2}$. At scale $\Theta(n^{1/2})$, we can simply apply the limit shape theorem to conclude that for almost all partitions, the sum of such column lengths is $\Theta(n)$, a constant fraction of the total size, because we are collecting a positive-area amount of the limit shape. Because almost all partitions have all columns of size $O(\sqrt{n})$, we have bootstrapped to completion with error probability $6\epsilon$ and so the result is proved.

\section{Appendix: Generalization of Dominance Result}

We extensively used Theorem~\ref{thm:dominance} stating that if $\varrho_m,\lambda$ are dominance comparable, then $c(\varrho_m,\varrho_m,\lambda)$. We present a generalization. 

\begin{thm}
\label{thm:gendom}

For partitions $\mu,$ $\nu \vdash n$, if $\mu$ has distinct row lengths and $\nu \succeq \mu$, then $c(\mu,\mu,\nu)$.

\end{thm}

This is strictly more general than Theorem~\ref{thm:dominance} because in the case when $\mu=\varrho_m$ we have $c(\mu,\mu,\nu)$ for all $\nu\succeq \mu$. Since $\mu=\mu'$, conjugating gives $c(\mu,\mu,\nu')$ as well, and so we have recovered both cases of Theorem~\ref{thm:dominance}.

\begin{proof}

The proof of this generalization requires only slight modification of the proof of the original result in \cite{IkenDom}. We quote from there the following definition. 

\begin{defn}
\label{defn:younghyper}
Let $d=|\lambda|=|\mu|=|\nu|$. A \textit{Young hypergraph} H of type $(\lambda,\mu,\nu)$, is a hypergraph with $d$ vertices such that 
\begin{enumerate}
\item The are three layers of hyperedges $E_{\lambda},E_{\mu},E_{\nu}$.
\item Each of $E_{\lambda},E_{\mu},E_{\nu}$ contains each vertex in exactly 1 hyperedge.
\item There is a bijection between the vertices of $H$ and the boxes of $\lambda$ such that 2 vertices lie in a common hyperedge of $E_{\lambda}$ iff the corresponding boxes in $\lambda$ lie in the same column. Analogously for $E_{\mu}$ and $\mu$ and for $E_{\nu}$ and $\nu$.
\end{enumerate}
\end{defn}

Given a Young hypergraph of type $(\lambda,\mu,\nu)$, we consider the ways to label its vertices with positive integers. In particular, we have the following definitions.

\begin{defn}
We call a labelling of the vertex set of a Young hypergraph of type $(\lambda,\mu,\nu)$ $\lambda$-\textit{permuting} if
for each hyperedge $e_{\lambda}$ of $\lambda$ with $|e_{\lambda}|=k$, the labels of the vertices of $e_{\lambda}$ are a permutation of $\{1,2,\dots,k\}$. We define $\mu$-permuting and $\nu$-permuting analogously.
\end{defn}

\begin{defn}
We call a labelling of the vertex set of a Young hypergraph of type $(\lambda,\mu,\nu)$ $\lambda$-\textit{distinct} if
for each hyperedge $e_{\lambda}$ of $\lambda$ with $|e_{\lambda}|=k$, the labels of the vertices of $e_{\lambda}$ consist of distinct numbers. We define $\mu$-distinct and $\nu$-distinct analogously.
\end{defn}

\begin{defn}

We call a labelling of the vertex set of a $(\lambda,\mu,\nu)$ Young hypergraph $H$ \textit{perfect} if it is $\lambda$-permuting, $\mu$-permuting, and $\nu$-distinct.
\end{defn}

Then the finish of the proof of \cite{IkenDom}[2.1] (see Section 5) amounts to the following lemma.

\begin{lem}
\label{lem:young}
If there exists a Young hypergraph such that there is exactly 1 perfect labelling of its vertices then $c(\lambda,\mu,\nu)$ holds.
\end{lem}

The proof in \cite{IkenDom} uses as Young hypergraphs for $\varrho_m$ the rows and columns of the Ferrers diagram of $\varrho_m$, which works because $\varrho_m$ is symmetric. We now adapt this to an arbitrary $\mu$ with distinct row lengths; for any $\nu\succeq \mu$ we show there exists a perfect Young hypergraph of type $(\mu,\mu,\nu)$. We use the Ferrers diagram for $\mu$ as the vertex set. For the first hypergraph $H_1$ corresponding to $\mu$, we simply take the columns of $\mu$ as the hyperedges. Note that in English coordinates this amounts to, for each hyperedge, greedily taking the left-most vertex in each available row. For our second hypergraph $H_2$ corresponding to $\mu$ we greedily take vertices from the right of each row instead. For example, when $\mu=\varrho_m$ this amounts to taking the diagonals as hyperedges. It is easy to see that $H_2$ is also a hypergraph of type $\mu$. We claim that $H_1$ and $H_2$ already limit the possible vertex labellings to a unique one, namely the labelling which assigns to each vertex its row number. 

\begin{lem}
There is only one vertex labelling of the Ferrers diagram of $\mu$ which is $\mu$-permuting with respect to both $H_1$ and $H_2$. This labelling assigns to each vertex its row number.
\end{lem}

\begin{proof}
We first make an easy observation. For any vertex $v$ of our hypergraph, the hyperedge of $H_2$ containing $v$ also contains a vertex in each higher row, all of which are strictly to the right of $v$. This is clear because the rows of $\mu$ have distinct lengths.

Now we inductively show that any such labelling must just be the row numbering. We induct on columns, starting from the furthest right. Clearly the rightmost column's unique vertex is labelled 1, because it is contained in a size 1 hyperedge of $H_1$.

For each subsequent column, assume that all columns to the right are labelled by row numbers. We do a further induction within the column, starting from the bottom vertex. Call the column under consideration $C$, and say its vertices are $v_1$ in row 1, and so on to $v_k$ in row k. For some vertex $v_i$ in $C$, assume that $v_j$ is labelled $j$ for all $j>i$; we show $v_i$ is labelled as $i$. 

To do so, note that by our initial observation, $v_i$ must be labelled at least $i$; its hyperedge in $H_2$ contains labels from $1$ to $(i-1)$ already. However, its hyperedge in $H_1$ is simply $C$, which already contains all labels greater than $i$. Therefore the only choice for $v_i$ is to be labelled $i$. 

This completes the induction. Clearly the described labelling indeed satisfies the given conditions, so the lemma is proved.
\end{proof}

Now the combination of the above lemmas implies that if we can find a perfect $(\mu,\mu,\nu)$ Young hypergraph for the above vertex labelling extending $H_1, H_2$ above, then we have $c(\mu,\mu,\nu)$. In fact, we can do so for precisely those $\nu$ which dominate $\mu$. It is clear that to find a $E_{\nu}$ hypergraph which yields a perfect Young hypergraph, it is equivalent to find a filling of $\nu$ with content $\mu$ such that each column has distinct entries. (By such a filling of $\nu$ with content $\mu$ we mean a labelling of the boxes of $\nu$ such that the number of $k$ labels is the size $\mu_k$ of the $k$th row of $\mu$.) By \cite{IkenDom}[4.1], the existence of such a filling is equivalent to $\nu \succeq \mu$, so the proof of Theorem~\ref{thm:gendom} is complete.
\end{proof}

\section{Appendix: The Generalized Semigroup Property and Tensor Cubes}

\subsection{The Semigroup Property for Many Partitions}
\label{subsec:longsemigroup}

We first generalize $c(\cdot,\cdot,\cdot)$ to longer sequences of representations.

\begin{defn}

For $k$ a positive integer let $$c(\lambda_1,\lambda_2,\dots,\lambda_k)$$
denote the assertion that $\lambda_1$ is a constituent of $\lambda_2\otimes \lambda_3\otimes \dots \lambda_k.$
\end{defn}

As in the $k=3$ case, $c$ is symmetric because it simply asserts the positivity of  $$\frac{1}{n!}\sum_{\sigma \in S_n} {\chi^{\lambda_1}(\sigma)}{\chi^{\lambda_2}(\sigma)}\dots {\chi^{\lambda_k}(\sigma)}.$$ We now show that the semigroup property still applies for longer sequences using induction.

\begin{lem}

\label{thm:longhsum}
If $c(\lambda_1, \dots \lambda_k)$ and $c(\mu_1,\dots \mu_k)$ then also $c(\lambda_1\Hplus\mu_1, \dots, \lambda_k \Hplus \mu_k).$ 
\end{lem}

\begin{proof}

First, the result is trivially true for $k\leq 2$ since $c(\lambda_1)=1 \iff \lambda_1$ is trivial and $c(\lambda_1,\lambda_2) \iff \lambda_1=\lambda_2$. For $k\geq3$ we induct from base case of \ref{thm:hsum}. So take $k>3$ and assume the result for $k-1$.\\

Because $c(\lambda_1, \dots \lambda_k)$ there exists a partition $\alpha$ such that $c(\lambda_1,\lambda_2,\dots \lambda_{(k-2)},\alpha)$ and $c(\lambda_{(k-1)},\lambda_k,\alpha)$. Similarly there exists $\beta$ such that $c(\mu_1,\dots, \mu_{(k-2)},\beta)$ and $c(\mu_{(k-1)},\mu_k,\beta)$. We now conclude using the inductive hypothesis that 

$$c(\lambda_1\Hplus\mu_1, \dots, \lambda_{(k-2)} \Hplus \mu_{(k-2)},\alpha\Hplus\beta)$$
and 
$$c(\lambda_{(k-1)} \Hplus \mu_{(k-1)},\lambda_k \Hplus \mu_k,\alpha\Hplus\beta).$$

Together these imply $c(\lambda_1\Hplus\mu_1, \dots, \lambda_k \Hplus \mu_k)$ as desired.

\end{proof}

As in the $k=3$ case, we may vertically add any even number of partitions in applying the semigroup property: this is because conjugating an even number of the partitions does not change the truth of $c(\lambda_1,\dots \lambda_k)$.

\subsection{Rectangles Appear in $\varrho_m^{\otimes 3}$}
\label{subsec:rectcube}

As mentioned in the remarks, rectangles are difficult to control using the semigroup property because they can only be broken up into smaller rectangles. This suggests that rectangles should be the hardest case for the Saxl conjecture. In this section, we show that despite this, rectangles appear in the tensor cube of the staircase. In fact, the proof is a fairly simple induction.

\begin{defn}
The rectangle partition $R(a,b)$ is the rectangular Young diagram $(a,a,\dots a)$ with $b$ rows.
\end{defn}

\begin{thm}
\label{thm:rectcube}
Any rectangular partition $\lambda=R(a,b)$ of size $\binom{m+1}{2}$ is a constituent in the tensor cube $\varrho_m^{\otimes 3}$. Equivalently, if $ab=\binom{m+1}{2}$ then $c(\varrho_m,\varrho_m,\varrho_m,R(a,b))$.
\end{thm}

\begin{proof}

We induct on $m$. Assume WLOG that $a\geq b$. If $a\geq m$ then $b \leq \frac{m+1}{2}$ and by Lemma~\ref{lem:domheights}, $\lambda$ and $\varrho_m$ are dominance comparable. Thus, the result follows in this case. 

Thus, we may assume that $\frac{m+1}{2}<a<m$. Then we have that $(2a-m)$ and $ (2m-a-1)$ are positive. We prove some simple lemmas.

\begin{lem}
\label{lem:stupid}
$b+2a-2m-1\geq 0$ and $a(b+2a-m-1)=|\varrho_{2a-m-1}|$.
\end{lem}

\begin{proof}
$0\leq \frac{(2a-m)(2a-m-1)}{2}=\frac{m^2+m}{2}+2a^2-2am-a=ab+2a^2-2am-a=a(b+2a-2m-1)$. Because $a\geq 0$, $b+2a-2m-1\geq 0$ follows.
\end{proof}

\begin{lem} If $\mu=R(2a-m,2m-a-1)$ then
$$c(\mu,\mu,\mu,\mu).$$
\end{lem}

\begin{proof}
In fact this holds for any $\mu$ at all.
We have $$k(\mu,\mu,\mu,\mu)=\langle \mu^{\otimes 2},\mu^{\otimes 2}\rangle>0$$
\end{proof}

\begin{lem}
$$c(\varrho_{(2m-2a)},\varrho_{(2m-2a)},\varrho_{(2m-2a)},R(m-a,2m-2a+1)).$$

\end{lem}

\begin{proof}
This follows from the inductive hypothesis because the total number of blocks in the rectangle and staircases are clearly equal.
\end{proof}

\begin{lem}
$$c\left(\varrho_{(2a-m-1)}, \varrho_{(2a-m-1)}, \varrho_{(2a-m-1)}, R(a,b+2a-2m-1)\right).$$

\end{lem}

\begin{proof}
This follows from the inductive hypothesis and Lemma~\ref{lem:stupid}.
\end{proof}

Now note that $$(R(2a-m,2m-2a+1)\Hplus R(m-a,2m-2a+1))\Vplus R(a,b+2a-2m+1)=$$
$$R(a,2m-2a-1)\Vplus R(a,b+2a-2m-1)=R(a,b)$$

while $$(R(2a-m,2m-2a+1)\Hplus \varrho_{(2m-a)})\Vplus \varrho_{(2a-m-1)}=\varrho_m.$$

This latter identity is a rewriting of the geometrically obvious $$(R(x,y)\Hplus \varrho_y )\Vplus \varrho_x =\varrho_{(x+y)}.$$ Theorem~\ref{thm:longhsum} applied to these lemmas and identities gives the result; because 4 is even, it is permissible to vertically add all 4 partitions when using the semigroup property. Geometrically, we are simply combining shapes as below.

\begin{figure}[h]
\label{fig:RectCube}
\caption{A Geometric Proof that Rectangles are Contained in $\varrho_m^{\otimes 3}$}
\includegraphics[scale=0.8]{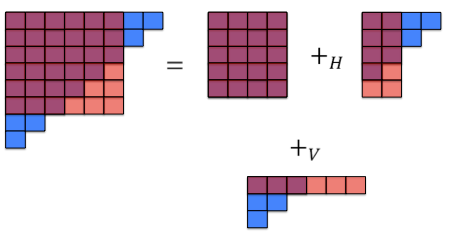}  
\end{figure}

\end{proof}

\bibliographystyle{acm}

\bibliography{bibliography}

\def\cprime{$'$}
\begin{thebibliography}{10}

\bibitem{Baik}
{\sc Baik, J., Deift, P., and Johansson, K.}
\newblock On the distribution of the length of the longest increasing
  subsequence of random permutations.
\newblock {\em J. Amer. Math. Soc. 12}, 4 (1999), 1119--1178.

\bibitem{symcube}
{\sc Bessenrodt, C., and Behns, C.}
\newblock On the {D}urfee size of {K}ronecker products of characters of the
  symmetric group and its double covers.
\newblock {\em J. Algebra 280}, 1 (2004), 132--144.

\bibitem{Prob}
{\sc Billingsley, P.}
\newblock {\em Probability and Measure}, vol.~939.
\newblock John Wiley \& Sons, 2012.

\bibitem{Borodin}
{\sc Borodin, A., Okounkov, A., and Olshanski, G.}
\newblock Asymptotics of {P}lancherel measures for symmetric groups.
\newblock {\em J. Amer. Math. Soc. 13}, 3 (2000), 481--515 (electronic).

\bibitem{IkenHard}
{\sc B{\"u}rgisser, P., and Ikenmeyer, C.}
\newblock The complexity of computing {K}ronecker coefficients.
\newblock In {\em 20th {A}nnual {I}nternational {C}onference on {F}ormal
  {P}ower {S}eries and {A}lgebraic {C}ombinatorics ({FPSAC} 2008)}, Discrete
  Math. Theor. Comput. Sci. Proc., AJ. Assoc. Discrete Math. Theor. Comput.
  Sci., Nancy, 2008, pp.~357--368.

\bibitem{semigp}
{\sc Christandl, M., Harrow, A.~W., and Mitchison, G.}
\newblock Nonzero kronecker coefficients and what they tell us about spectra.
\newblock {\em Communications in mathematical physics 270}, 3 (2007), 575--585.

\bibitem{Unif1s}
{\sc Fristedt, B.}
\newblock The structure of random partitions of large integers.
\newblock {\em Trans. Amer. Math. Soc. 337}, 2 (1993), 703--735.

\bibitem{Fulman}
{\sc Fulman, J.}
\newblock Stein's method and {P}lancherel measure of the symmetric group.
\newblock {\em Trans. Amer. Math. Soc. 357}, 2 (2005), 555--570.

\bibitem{FH}
{\sc Fulton, W., and Harris, J.}
\newblock {\em Representation theory}, vol.~129.
\newblock Springer Science \& Business Media, 1991.

\bibitem{IkenDom}
{\sc Ikenmeyer, C.}
\newblock The {S}axl conjecture and the dominance order.
\newblock {\em Discrete Math. 338}, 11 (2015), 1970--1975.

\bibitem{bess}
{\sc Krasikov, I.}
\newblock Uniform bounds for {B}essel functions.
\newblock {\em J. Appl. Anal. 12}, 1 (2006), 83--91.

\bibitem{pak}
{\sc Pak, I., Panova, G., and Vallejo, E.}
\newblock Kronecker products, characters, partitions, and the tensor square
  conjectures.
\newblock {\em arXiv preprint arXiv:1304.0738\/} (2013).

\bibitem{quote}
{\sc Regev, A.}
\newblock Kronecker multiplicities in the {$(k,\ell)$} hook are polynomially
  bounded.
\newblock {\em Israel J. Math. 200}, 1 (2014), 39--48.

\bibitem{enumcombo}
{\sc Stanley, R.~P.}
\newblock {\em Enumerative combinatorics. {V}ol. 2}, vol.~62 of {\em Cambridge
  Studies in Advanced Mathematics}.
\newblock Cambridge University Press, Cambridge, 1999.
\newblock With a foreword by Gian-Carlo Rota and appendix 1 by Sergey Fomin.

\bibitem{UnifLimShape}
{\sc Vershik, A.~M.}
\newblock Statistical mechanics of combinatorial partitions, and their limit
  configurations.
\newblock {\em Funktsional. Anal. i Prilozhen. 30}, 2 (1996), 19--39, 96.

\bibitem{PlanchLimShape}
{\sc Ver{\v{s}}ik, A.~M., and Kerov, S.~V.}
\newblock Asymptotic behavior of the {P}lancherel measure of the symmetric
  group and the limit form of {Y}oung tableaux.
\newblock {\em Dokl. Akad. Nauk SSSR 233}, 6 (1977), 1024--1027.

\end{thebibliography}

\end{document}